\documentclass{amsart}
\usepackage{amssymb, amsmath}
\usepackage{bbm}
\newtheorem{theorem}{Theorem}[section]
\newtheorem{lemma}[theorem]{Lemma}
\newtheorem{proposition}[theorem]{Proposition}

\newtheorem{claim}[theorem]{Claim}
\newtheorem{subclaim}[theorem]{Subclaim}
\newtheorem{fact}[theorem]{Fact}

\newenvironment{definition}[1][Definition]{\begin{trivlist}
\item[\hskip \labelsep {\bfseries #1}]}{\end{trivlist}}

\begin{document}
\title{Good and Bad Points in Scales}
\author{Chris Lambie-Hanson}
\address{Department of Mathematical Sciences, Carnegie Mellon University \\ Pittsburgh, PA 15213}
\email{clambieh@andrew.cmu.edu}
\thanks{This work will form a part of the author's Ph.D. thesis, completed under the supervision of James Cummings, whom the author would like to thank for many helpful conversations.}
\date{\today}
\begin{abstract}
	We address three questions raised by Cummings and Foreman regarding a model of Gitik and Sharon. We first analyze the PCF-theoretic structure of the Gitik-Sharon model, determining the extent of good and bad scales. We then classify the bad points of the bad scales existing in both the Gitik-Sharon model and other models containing bad scales. Finally, we investigate the ideal of subsets of singular cardinals of countable cofinality carrying good scales.
\end{abstract}
\maketitle

\section{Introduction}

The study of singular cardinals and their successors has become of central importance in combinatorial set theory and is intimately related to questions regarding large cardinals, inner model theory, and cardinal arithmetic. One of the most useful tools at our disposal for the study of singulars and their successors is Shelah's PCF theory, in which the investigation of the cofinalities of reduced products of regular cardinals has been used to obtain a number of remarkable results (see \cite{abrahammagidor} for a good introduction to the topic). In this paper, we analyze the PCF-theoretic structure of a model of Gitik and Sharon \cite{gitiksharon}. We then conclude with a few results about the ideal of sets that carry good scales.  

Our notation is for the most part standard. Unless otherwise stated, \cite{jech} is our reference for notation and terminology. If $A$ is a set of cardinals, then $\prod A$ is the set of functions $f$ such that $\mathrm{dom}(f) = A$ and, for every $\lambda \in A$, $f(\lambda) \in \lambda$. If $A = \{\kappa_i \mid i < \eta \}$, we will often write $\prod_{i<\eta} \kappa_i$ instead of $\prod A$ and will write $f(i)$ instead of $f(\kappa_i)$. If $A$ has no maximum element and $f,g \in \prod A$, then we write $f <^* g$ to mean that there is $\gamma \in A$ such that, for every $\lambda \in A\setminus \gamma$, $f(\lambda) < g(\lambda)$. If $\lambda$ is a regular cardinal, then $\mathrm{cof}(\lambda)$ is the class of ordinals $\alpha$ such that $\mathrm{cf}(\alpha)=\lambda$. Expressions such as $\mathrm{cof}(>\lambda)$ are defined in the obvious way. If $x$ is a well-ordered set, then $\mathrm{otp}(x)$ is the order type of $x$. If $\kappa \leq \lambda$ are cardinals, then $\mathcal{P}_\kappa(\lambda) = \{x \subset \lambda \mid |x| < \kappa \}$. If $x,y \in \mathcal{P}_\kappa(\lambda)$ are such that $x\cap \kappa \in \kappa$ and $y\cap \kappa \in \kappa$, then we write $x \prec y$ if $x\subseteq y$ and $\mathrm{otp}(x) < y \cap \kappa$. 

We start by recalling some basic definitions.

\begin{definition}
If $A$ is a set of regular cardinals, then $A$ is {\it progressive} if $|A| < \min(A)$.
\end{definition}

For technical reasons, we assume throughout this paper that we are working with progressive sets of regular cardinals.

\begin{definition}
Suppose $\kappa$ is a singular cardinal and $A$ is a cofinal subet of $\kappa$ consisting of regular cardinals. $\overrightarrow{f} = \langle f_\alpha \mid \alpha < \mu \rangle$ is called a {\it scale of length $\mu$ in $\prod A$} if the following hold:
\begin{enumerate}
\item{For all $\alpha < \mu$, $f_\alpha \in \prod A$.}
\item{For all $\alpha < \beta < \mu$, $f_\alpha <^* f_\beta$.}
\item{For all $h \in \prod A$, there is $\alpha < \mu$ such that $h<^* f_\alpha$.}
\end{enumerate}
In other words, $\overrightarrow{f}$ is increasing and cofinal in $(\prod A, <^*)$. We say that $A$ {\it carries a scale of length $\mu$} if there is a scale of length $\mu$ in $\prod A$.
\end{definition}

We note that we will typically consider scales in $\prod A$ when $\mathrm{otp}(A) = \mathrm{cf}(\kappa)$, but this need not necessarily be the case. A simple diagonalization argument shows that, if $2^\kappa = \kappa^+$, then every cofinal $A\subset \kappa$ consisting of regular cardinals carries a scale of length $\kappa^+$. One of the fundamental results of PCF theory is a theorem of Shelah stating that for every singular cardinal $\kappa$, there is an $A\subseteq \kappa$ that carries a scale of length $\kappa^+$ \cite{shelah}.

\begin{definition}
Let $\kappa$ be a singular cardinal, $A\subset \kappa$ a cofinal set of regular cardinals, and $\overrightarrow{f} = \langle f_\alpha \mid \alpha < \mu \rangle$ a scale in $\prod A$.
\begin{enumerate}
\item{$\alpha < \mu$ is called a {\it good point for $\overrightarrow{f}$ (very good point for $\overrightarrow{f}$)} if $\mathrm{cf}(\kappa) < \mathrm{cf}(\alpha) < \kappa$ and there are an unbounded (club) $C\subseteq \alpha$ and $\eta < \kappa$ such that, for all $\gamma< \gamma'$, both in $C$, $f_\gamma \restriction (A\setminus \eta) < f_{\gamma'} \restriction (A\setminus \eta)$.}
\item{$\alpha < \mu$ is called a {\it bad point for $\overrightarrow{f}$} if $\mathrm{cf}(\kappa) < \mathrm{cf}(\alpha) < \kappa$ and $\alpha$ is not a good point for $\overrightarrow{f}$.}
\item{$\overrightarrow{f}$ is a {\it good scale (very good scale)} if $\mu = \kappa^+$ and there is a club $C \subseteq \kappa^+$ such that every $\alpha \in C \cap \mathrm{cof}(>\mathrm{cf}(\kappa))$ is a good point (very good point) for $\overrightarrow{f}$.}
\item{$\overrightarrow{f}$ is a {\it bad scale} if $\mu = \kappa^+$ and $\overrightarrow{f}$ is not a good scale.}
\end{enumerate}
\end{definition}

An intricate web of implications connects the existence of good and very good scales with various other combinatorial principles, including squares, approachability, and stationary reflection, at successors of singular cardinals. We record some of the relevant facts here. Let $\kappa$ be a singular cardinal.

\begin{itemize}
\item{If $A\subseteq \kappa$ carries a good scale, then every scale in $\prod A$ is good \cite{foremanmagidor}.}
\item{If $\square_\kappa$ holds, then every $A\subseteq \kappa$ which carries a scale of length $\kappa^+$ carries a very good scale \cite{cfm}.}
\item{If $AP_\kappa$ holds, then every $A\subseteq \kappa$ which carries a scale of length $\kappa^+$ carries a good scale \cite{foremanmagidor}.}
\end{itemize}

The interested reader is referred to \cite{cfm} and \cite{foremanmagidor} for more details.

Woodin asked whether the failure of SCH at $\kappa$ implies that $\square^*_\kappa$ holds, and Cummings, Foreman, and Magidor \cite{cfm} asked whether the existence of a very good scale of length $\kappa^+$ implies that $\square^*_\kappa$ holds. Gitik and Sharon, in \cite{gitiksharon}, answer both of these questions by producing, starting with a supercompact cardinal, a model in which there is a singular strong limit cardinal $\kappa$ of cofinality $\omega$ such that $2^\kappa = \kappa^{++}$, $AP_\kappa$ fails (and hence $\square^*_\kappa$ fails), and there is an $A\subseteq \kappa$ that carries a very good scale. 

Cummings and Foreman, in \cite{cummingsforeman}, show that, in the model of \cite{gitiksharon}, there is a $B\subseteq \kappa$ that carries a bad scale, thus providing another proof of the failure of $AP_\kappa$. Cummings and Foreman go on to raise a number of other questions, three of which we address in this paper:

\begin{enumerate}
\item{Do there exist any other interesting scales in the model of \cite{gitiksharon}?}
\item{Into which case of Shelah's Trichotomy Theorem do the bad points of the bad scale in \cite{cummingsforeman} fall?}
\item{When the first PCF generator exists, does it have a maximal (modulo bounded subsets of $\kappa$) subset which carries a good scale?}
\end{enumerate}

\section{Diagonal Supercompact Prikry Forcing} \label{priksec}

We review here some key facts from Gitik and Sharon's construction in \cite{gitiksharon}. At the heart of their argument is a diagonal version of supercompact Prikry forcing.

Let $\kappa$ be a supercompact cardinal, let $\mu = \kappa^{+\omega+1}$, and let $U$ be a normal, fine ultrafilter over $\mathcal{P}_\kappa (\mu)$. For $n<\omega$, let $U_n$ be the projection of $U$ on $\mathcal{P}_\kappa(\kappa^{+n})$, i.e. $X \in U_n$ if and only if $\{y \in \mathcal{P}(\kappa^{+\omega+1}) \mid y\cap \kappa^{+n} \in X \} \in U$. Note that each $U_n$ is a normal, fine ultrafilter over $\mathcal{P}_\kappa(\kappa^{+n})$ concentrating on the set $X_n = \{x\in \mathcal{P}_\kappa(\kappa^{+n}) \mid \kappa_x := x\cap \kappa \mbox{ is an inaccessible cardinal and, for all } i \leq n, \mathrm{otp}(x\cap \kappa^{+i}) = \kappa_x^{+i} \}$.

We are now ready to define the diagonal supercompact Prikry forcing, $\mathbb{Q}$. Conditions of $\mathbb{Q}$ are of the form $q = \langle x_0^q, x_1^q, \ldots , x_{n-1}^q, A_n^q, A_{n+1}^q, \ldots \rangle$, where 
\begin{enumerate}
\item{For all $i<n$, $x_i^q \in X_i$.}
\item{For all $i\geq n$, $A_i^q \in U_i$.}
\item{For all $i<j<n$, $x^q_i \prec x^q_{i+1}$.}
\item{For all $i<n$, $j\geq n$, and $y\in A_j^q$, $x^q_i \prec y$.}
\end{enumerate}
$n$ is called the {\it length} of $q$ and is denoted lh($q$). $\langle x_0^q, \ldots ,x_{n-1}^q \rangle$, denoted $s(q)$, is called the {\it lower part} of $q$, and $\langle A_n^q, A_{n+1}^q,\ldots \rangle$ is called the {\it upper part} of $q$. If $s$ is a lower part of length $n$, let $s{^\frown}\mathbbm{1}$ denote the condition $s{^\frown}\langle X_n, X_{n+1}, \ldots \rangle$. If $s = \langle x_0, \ldots ,x_{n-1} \rangle$ is a lower part and $q$ is a condition of the form $\langle B_0, B_1, \ldots \rangle$ such that, for every $i < n$, $x_i \in B_i$, then $s ^\frown q$ denotes the condition $\langle x_0, \ldots ,x_{n-1}$, $B'_n, B'_{n+1} \ldots \rangle$, where, for $i \geq n$, $B'_i = \{y \in B_i \mid x_{n-1} \prec y \}$.

For $p, q \in \mathbb{Q}$, $p \leq q$ if and only if 
\begin{enumerate}
\item{lh($p$) $\geq$ lh($q$).}
\item{For all $i<$lh($q$), $x_i^p = x_i^q$.}
\item{For all $i$ such that lh($q) \leq i <$lh($p$), $x_i^p \in A_i^q$.}
\item{For all $i\geq$lh($p$), $A_i^p \subseteq A_i^q$.}
\end{enumerate}
We say $p$ is a {\it direct extension of }$q$, and write $p \leq^* q$, if $p \leq q$ and lh($p$) = lh($q$).

We now summarize some relevant facts about $\mathbb{Q}$. The reader is referred to \cite{gitiksharon} for proofs.
\begin{itemize}
\item{(Diagonal intersection) Suppose that, for every lower part $s$, $A_s$ is an upper part such that $s{^\frown}A_s \in \mathbb{Q}$. Then there is a sequence $\langle B_n \mid n<\omega \rangle$ such that, for every $n$, $B_n \in U_n$ and, for every lower part $s$ of length $n$, every extension of $s{^\frown} \langle B_i \mid i\geq n \rangle$ is compatible with $s{^\frown}A_s$.}
\item{(Prikry property) Let $q\in \mathbb{Q}$ and let $\phi$ be a statement in the forcing language. Then there is $p\leq^* q$ such that $p \parallel \phi$. In particular, $\mathbb{Q}$ adds no new bounded subsets of $\kappa$.}
\item{The generic object added by $\mathbb{Q}$ is an $\omega$-sequence $\langle x_n \mid n<\omega \rangle$, where, for all $n<\omega$, $x_n \in X_n$ and $x_n \prec x_{n+1}$. Letting $\kappa_n = \kappa_{x_n}$, $\langle \kappa_n \mid n<\omega \rangle$ is cofinal in $\kappa$, so $\mathrm{cf}(\kappa)^{V[G]}=\omega$. $\bigcup_{n<\omega}x_n = \kappa^{+\omega}$, so, in $V[G]$, for all $i\leq \omega$, $\kappa^{+i}$ is an ordinal of cofinality $\omega$ and size $\kappa$.}
\item{Any two conditions with the same lower part are compatible. In particular, since there are only $\kappa^{+\omega}$-many lower parts, $\mathbb{Q}$ satisfies the $\kappa^{+\omega+1}$-c.c. Thus, $\kappa^{+\omega+1} = \mu$ is preserved in the extension, and $(\kappa^+)^{V[G]} = \mu$.}
\item{Let $\langle A_n \mid n<\omega \rangle$ be such that, for each $n<\omega$, $A_n \in U_n$. Then there is $n^*$ such that, for all $n\geq n^*$, $x_n \in A_n$.}
\item{If $A\in V[G]$ is a set of ordinals such that $\mathrm{otp}(A) = \nu$, where $\omega < \nu = \mathrm{cf}^V(\nu) < \kappa$, then there is an unbounded subset $B$ of $A$ such that $B\in V$.}
\end{itemize}

\section{Scales in the Gitik-Sharon Model}
In \cite{gitiksharon}, Gitik and Sharon obtain their desired model by starting with a supercompact cardinal, $\kappa$, performing an Easton-support iteration to make $2^{\kappa} = \kappa^{+\omega+2}$ while preserving the supercompactness of $\kappa$, and, in the resulting model, forcing with $\mathbb{Q}$. They then show that, in the final model, there is a very good scale in $\prod_{n<\omega} \kappa_n^{+\omega+1}$. We show that, with a bit more care in preparing the ground model, we can arrange so that there are other scales with many very good points.

Let $\kappa$ be supercompact, and suppose that GCH holds. Let $\mu = \kappa^{+\omega+1}$, let $U$ be a supercompactness measure on $\mathcal{P}_\kappa(\mu)$, and let $j:V\rightarrow M\cong Ult(V,U)$. If $\lambda$ is a regular cardinal, let $\mathbb{A}(\lambda)$ denote the full-support product of $\mathrm{Add}(\lambda^{+n}, \lambda^{+\omega+2})$ for $n<\omega$, where $\mathrm{Add}(\lambda^{+n}, \lambda^{+\omega+2})$ is the poset whose conditions are functions $f$ such that $\mathrm{dom}(f)\subseteq \lambda^{+\omega+2}$, $|\mathrm{dom}(f)|<\lambda^{+n}$, and, for every $\alpha \in \mathrm{dom}(f)$, $f(\alpha)$ is a partial function from $\lambda^{+n}$ to $\lambda^{+n}$ of size less than $\lambda^{+n}$. If $p=\langle p_n \mid n<\omega \rangle$ is a condition in $\mathbb{A}(\lambda)$ and $\alpha < \lambda^{+\omega+2}$, denote by $p\restriction \alpha$ the condition $\langle p_n\restriction \alpha \mid n<\omega \rangle$, and let $\mathbb{A}(\lambda)\restriction \alpha = \{p\restriction \alpha \mid p\in \mathbb{A}(\lambda)\}$. Let $\mathbb{P}$ denote the iteration with backward Easton support of $\mathbb{A}(\lambda)$ for inaccessible $\lambda \leq \kappa$. For each $\lambda$, let $\mathbb{P}_{<\lambda}$ denote the iteration below $\lambda$ and let $\mathbb{P}_\lambda$ be $\mathbb{P}_{<\lambda}*\mathbb{A}(\lambda)$. Let $G$ be $\mathbb{P}$-generic over $V$.

\begin{lemma} \label{preplem}
In $V[G]$, we can extend $j$ to $j^*$, a $\mu$-supercompactness embedding with domain $V[G]$, such that for every $n<\omega$ and every $\beta<j(\kappa^{+n})$, there is $g^n_\beta:\kappa^{+n}\rightarrow \kappa^{+n}$ such that $j^*(g^n_\beta)(\sup(j``\kappa^{+n}))=\beta$.
\end{lemma}

\begin{proof}
First note that, in V, for every $\alpha \leq \mu^+$, $|j(\alpha)| \leq |{^{\mathcal{P}_\kappa(\mu)}\alpha}| \leq \mu^+$, and, since ${^{\mu}M}\subseteq M$, $\mathrm{cf}(j(\lambda))=\mu^+$ for every regular $\lambda$ with $\kappa \leq \lambda \leq \mu^+$.. Thus, the number of antichains of $j(\mathbb{P}_{<\kappa})/G$ in $M[G]$ is $\mu^+$ and, since ${^{\mu}M[G]}\subseteq M[G]$ and $j(\mathbb{P}_{<\kappa})/G$ is $\mu^+$-closed, we can find $H\in V[G]$ that is $j(\mathbb{P}_{<\kappa})/G$-generic over $M[G]$. Let $G_\kappa=\langle f^n_\alpha \mid n<\omega, \ \alpha<\mu^+ \rangle$ be the generic object for $\mathbb{A}(\kappa)$ added by $G$. For $n<\omega$, let $\langle \delta^n_\alpha \mid \alpha < \mu^+ \rangle$ enumerate $j(\kappa^{+n})$.

As before, we can find, in $V[G]$, an $I^*$ which is $j(\mathbb{A}(\kappa))=\mathbb{A}(j(\kappa))^{M[G*H]}$-generic over $M[G*H]$. For $\alpha < j(\mu^+)$, let $I^*\restriction \alpha = \{p\restriction \alpha \mid p\in I^*\}$. Note that $I^*\restriction \alpha$ is $j(\mathbb{A}(\kappa))\restriction \alpha$-generic over $M[G*H]$. For each $\alpha < j(\mu^+)$, let $I\restriction \alpha$ be formed by minimally adjusting $I^*\restriction \alpha$ so that, for every $p\in I$, $n<\omega$, and $\eta < \mu^+$, if $j(\eta)<\alpha$ and $j(\eta)\in \mathrm{dom}(p_n)$, then $p_n(j(\eta))$ is compatible with $j``f^n_\eta$ and $p_n(j(\eta))(\sup(j``\kappa^{+n}))=\delta^n_\eta$. Since $j``\mu^+$ is cofinal in $j(\mu^+)$, the number of changes to each condition is at most $\mu$, so each adjusted $p$ is itself in $M[G*H]$ and $I\restriction \alpha$ is $j(\mathbb{A}(\kappa))\restriction \alpha$-generic over $M[G*H]$. Let 
\[
I=\bigcup_{\alpha<j(\mu^+)}I\restriction \alpha.
\]
By chain condition, every maximal antichain of $j(\mathbb{A}(\kappa))$ is a subset of $j(\mathbb{A}(\kappa))\restriction \alpha$ for some $\alpha < j(\mu^+)$, so $I$ is $j(\mathbb{A}(\kappa))$-generic over $M[G*H]$. Now $j``G \subseteq G*H*I$, so we can lift $j$ to $j^*$ with domain $V[G]$ and $j(G)=G*H*I$. By construction, for every $n<\omega$ and $\alpha<\mu^+$, $j^*(f^n_\alpha)(\sup(j``\kappa^{+n}))=\delta^n_\alpha$, so, for $\beta < j(\kappa^{+n})$, letting $g^n_\beta = f^n_\alpha$, where $\delta^n_\alpha = \beta$, gives $j^*$ the desired properties.
\end{proof}

Let $U^*$ be the measure on $\mathcal{P}_\kappa(\mu)$ derived from $j^*$, and, for $n<\omega$, let $U_n^*$ be the projection of $U^*$ onto $\mathcal{P}_\kappa(\kappa^{+n})$ and $j^*_{U^*_n}$ be the embedding derived from $U^*_n$. Note that $U_n^* = \{X\subseteq \mathcal{P}_\kappa(\kappa^{+n}) \mid j``\kappa^{+n} \in j^*(X) \}$. Also note that, for all $n < \omega$, the functions $\langle g^n_\alpha \mid \alpha<j(\kappa^{+n}) \rangle$ witness that $j^*_n(\kappa^{+n})=j(\kappa^{+n})$. Let $\mathbb{Q}$ be the diagonal supercompact Prikry forcing defined using the $U_n^*$'s. Let $H = \langle x_n \mid n<\omega \rangle$ be $\mathbb{Q}$-generic over $V[G]$, and let $\kappa_n = x_n \cap \kappa$.

\begin{theorem} \label{vgthm}
In $V[G*H]$, there is a scale in 
\[
\prod_{\substack{n<\omega\\ i\leq n}}\kappa_{n+1}^{+i}
\]
of length $\mu^+$ such that every $\alpha < \mu^+$ with $\omega<\mathrm{cf}(\alpha)<\kappa$ is very good.
\end{theorem}

\begin{proof}
For each $n<\omega$, fix an increasing, continuous sequence of ordinals $\langle \alpha^n_\zeta \mid \zeta < \mu^+ \rangle$ cofinal in $j(\kappa^{+n})$. For all $\zeta<\mu^+$, $n<\omega$, and $i\leq n$, let $f_\zeta(n,i)=g^i_{\alpha^i_\zeta}(\sup(x_n \cap \kappa^{+i}))$.

\begin{claim}
Let $\zeta < \mu^+$. There is $n_\zeta < \omega$ such that for all $n\geq n_\zeta$ and all $i\leq n$, $f_\zeta(n,i) < \sup(x_{n+1}\cap \kappa^{+i})$.
\end{claim}

\begin{proof}
For $n<\omega$, let $A_{n+1} = \{x\in \mathcal{P}_\kappa(\kappa^{+n+1}) \mid$ for all $i\leq n$ and all $y\in \mathcal{P}_\kappa(\kappa^{+n})$ with $y\prec x$, $g^i_{\alpha^i_\zeta}(\sup(y \cap \kappa^{+i}))<\sup(x\cap \kappa^{+i})\}$. Now suppose that $\bar{y} \in \mathcal{P}_{j(\kappa)}(j(\kappa^{+n}))$ is such that $\bar{y} \prec j``\kappa^{+n+1}$. Then, since $j``\kappa^{+n+1} \cap j(\kappa) = \kappa$, $\mathrm{otp}(\bar{y})<\kappa$ and, since $\bar{y} \subseteq j``\kappa^{+n+1}$, if $y\in \mathcal{P}_\kappa(\kappa^{+n})$ is the inverse image of $\bar{y}$ under $j$, then $j^*(y)=\bar{y}$, so, for all $i\leq n$, $j^*(g^i_{\alpha^i_\zeta})(\sup(\bar{y} \cap j(\kappa^{+i}))) = j^*(g^i_{\alpha^i_\zeta}(\sup(y\cap \kappa^{+i}))) < \sup(j``\kappa^{+i})=\sup(j``\kappa^{+n+1}\cap j(\kappa^{+i}))$. Thus, $j``\kappa^{+n+1} \in j^*(A_{n+1})$, so $A_{n+1} \in U^*_{n+1}$. By genericity, there is $n_\zeta < \omega$ such that $x_{n+1}\in A_{n+1}$ for all $n\geq n_\zeta$. The claim follows.
\end{proof}

Thus, by adjusting each $f_\zeta$ on only finitely many coordinates, we may assume that, for all $\zeta < \mu^+$,
\[
f_\zeta \in \prod_{\substack{n<\omega\\ i\leq n}}\sup(x_{n+1}\cap \kappa^{+i}).
\]

\begin{claim}
For all $\zeta < \zeta' < \mu^+$, $f_\zeta <^* f_{\zeta'}$.
\end{claim}

\begin{proof}
Fix $\zeta < \zeta' < \mu^+$. For all $n<\omega$ and $i\leq n$, we have $\alpha^i_\zeta = j^*(g^i_{\alpha^i_\zeta})(\sup(j``\kappa^{+n} \cap j(\kappa^{+i}))) < j^*(g^i_{\alpha^i_{\zeta'}})(\sup(j``\kappa^{+n} \cap j(\kappa^{+i}))) = \alpha^i_{\zeta'}$. Thus, the set $B_n = \{x \in \mathcal{P}_\kappa(\kappa^{+n}) \mid$ for all $i\leq n$, $g^i_{\alpha^i_\zeta}(\sup(x\cap \kappa^{+i}))<g^i_{\alpha^i_{\zeta'}}(\sup(x\cap \kappa^{+i})) \}$ is in $U^*_n$. By genericity, $x_n \in B_n$ for large enough $n<\omega$, so, for large enough $n$, for all $i\leq n$, $f_\zeta(n,i) <f_{\zeta'}(n,i)$.
\end{proof}

\begin{lemma} \label{bdg1}
In $V[G*H]$, let 
\[
h \in \prod_{\substack{n<\omega\\ i\leq n}}\sup(x_{n+1}\cap \kappa^{+i}).
\]
Then there is $\langle H_{n,i} \mid n<\omega, \ i \leq n \rangle \in V[G]$ such that $H_{n,i}:\mathcal{P}_\kappa(\kappa^{+n})\rightarrow \kappa^{+i}$ and, for large enough $n$ and all $i\leq n$, $h(n,i) < H_{n,i}(x_n)$. 
\end{lemma}

\begin{proof}
Let $h$ be as in the statement of the lemma, and let $\dot{h}$ be a $\mathbb{Q}$-name for $h$. We may assume that, in fact, 
\[
h \in \prod_{\substack{n<\omega\\ i\leq n}}x_{n+1}\cap \kappa^{+i}
\]
by considering instead $h'$, where $h'(n,i)=\min(x_{n+1}\setminus h(n,i))$.

 We show that for every $q\in \mathbb{Q}$, there is $p \leq^* q$ forcing the desired conclusion. We assume for simplicity that $q$ is the trivial condition and that 
\[
q\Vdash ``\dot{h} \in \prod_{\substack{n<\omega\\ i\leq n}}\dot{x}_{n+1}\cap \kappa^{+i}".
\]
A tedious but straightforward adaptation of our proof gives the general case.

Work in $V[G]$. If $s$ is a lower part of length $n+2$ with maximum element $x^s_{n+1}$, then, for every $i\leq n$, $s^\frown \mathbbm{1} \Vdash ``\dot{h}(n,i) \in x^s_{n+1}"$. Since $|x^s_{n+1}|<\kappa$ and $(\mathbb{Q},\leq^*)$ is $\kappa$-closed, repeated application of the Prikry property yields an upper part $A_s$ and ordinals $\langle \alpha_{s,i} \mid i\leq n \rangle$ such that, for all $i\leq n$, $s^\frown A_s \Vdash ``\dot{h}(n,i) = \alpha_{s,i}"$. By taking a diagonal intersection, we obtain a condition $q' = \langle B_0, B_1, \ldots \rangle$ such that for every lower part $s$ of length $n+2$ compatible with $q'$ and every $i\leq n$, $s^\frown q' \Vdash ``\dot{h}(n,i) = \alpha_{s,i}"$. 

Now suppose $t$ is a lower part of length $n+1$ compatible with $q'$, and let $i\leq n$. Consider the regressive function with domain $B_{n+1}$ which takes $x$ and returns $\alpha_{t^\frown \langle x \rangle, i}$. By Fodor's lemma, this function is constant on a measure-one set $B_{t,i}$. Let $B_t = \bigcap_{i\leq n} B_{t,i}$. By taking the diagonal intersections of the $B_t$'s, we obtain a condition $p=\langle C_0, C_1, \ldots \rangle$ such that for every lower part $t$ of length $n+1$ compatible with $p$ and for every $i\leq n$, there is $\beta_{t,i}$ such that $t^\frown p \Vdash ``\dot{h}(n,i) = \beta_{t,i}"$.

Now, for $n<\omega$, $i\leq n$, and $x\in \mathcal{P}_\kappa(\kappa^{+n})$, let $H_{n,i}(x) = \sup(\{\beta_{t,i} + 1 \mid t$ is a lower part of length $n+1$ with top element $x\})$. Since there are fewer than $\kappa$-many such lower parts, it is clear that $H_{n,i}: \mathcal{P}_\kappa(\kappa^{+n}) \rightarrow \kappa^{+i}$ and, for all $n<\omega$ and $i\leq n$, $p \Vdash ``\dot{h}(n,i) < H_{n,i}(\dot{x}_n)"$.
\end{proof}

\begin{claim}
$\overrightarrow{f} = \langle f_\zeta \mid \zeta < \mu^+ \rangle$ is cofinal in
\[
\prod_{\substack{n<\omega\\ i\leq n}}\sup(x_{n+1}\cap \kappa^{+i}).
\]
\end{claim}

\begin{proof}
Let 
\[
h\in \prod_{\substack{n<\omega\\ i\leq n}}\sup(x_{n+1}\cap \kappa^{+i}).
\]
Find $\langle H_{n,i} \mid n<\omega, \ i\leq n \rangle \in V[G]$ as in the previous lemma. For each $n$ and $i$, $[H_{n,i}]_{U_n} < j(\kappa^{+i})$. For $i<\omega$, let $\alpha_i = \sup(\{[H_{n,i}]_{U_n}+1 \mid n \geq i \})$. For every $i<\omega$, $\alpha_i < j(\kappa^{+i})$. Find $\zeta < \mu^+$ such that, for all $i<\omega$, $\alpha_i < \alpha^i_\zeta$. For all $n<\omega$ and $i\leq n$, $[H_{n,i}]_{U_n} < \alpha^i_{\zeta} = j^*(g^i_{\alpha^i_\zeta})(\sup(j``\kappa^{+i}))$. Thus, the set of $x\in \mathcal{P}_\kappa(\kappa^{+n})$ such that $H_{n,i}(x) < g^i_{\alpha^i_\zeta}(\sup(x\cap \kappa^{+i}))$ is in $U_n$. By genericity, for large enough $n$ and all $i\leq n$, $h(n,i) < H_{n,i}(x_n) < g^i_{\alpha^i_\zeta}(\sup(x_n\cap \kappa^{+i})) = f_\zeta(n,i)$, so $\overrightarrow{f}$ is in fact cofinal in the desired product.
\end{proof}

\begin{claim}
Suppose $\alpha < \mu^+$ and $\omega < \mathrm{cf}(\alpha) < \kappa$ (in $V[G*H]$). Then $\alpha$ is very good for $\overrightarrow{f}$.
\end{claim}

\begin{proof}
Since $\omega < \mathrm{cf}(\alpha) < \kappa$ in $V[G*H]$, the same is true in $V[G]$, by the last line of Section \ref{priksec}. Let $C\in V[G]$ be a club in $\alpha$ with $\mathrm{otp}(C) = \mathrm{cf}(\alpha)$. Let $n<\omega$, $i\leq n$, and $\zeta < \zeta'$ with $\zeta, \zeta' \in C$. Then it is easy to see that $A_{n,i,\zeta, \zeta'} := \{x\in \mathcal{P}_\kappa(\kappa^{+n}) \mid g^i_{\alpha^i_\zeta}(\sup(x \cap \kappa^{+i})) < g^i_{\alpha^i_{\zeta'}}(\sup(x \cap \kappa^{+i})) \}$ is in $U^*_n$. Since $|C| < \kappa$ and $U^*_n$ is $\kappa$-complete, we get that 
\[
A_n = \bigcap_{\substack{i\leq n \\ \zeta<\zeta' \in C}} A_{n,i,\zeta,\zeta'} \in U^*_n.
\]
Thus, for large enough $n$, $x_n \in A_n$, so $C$ witnesses that $\alpha$ is very good.
\end{proof}
We now have a scale with all of the desired properties, except it lives in the wrong product. Notice, though, that for every $n<\omega$ and $i\leq n$, we have arranged that $\mathrm{cf}(\sup(x_{n+1}\cap \kappa^{+i})) = \kappa_{n+1}^{+i}$. Thus, through standard arguments, $\overrightarrow{f}$ collapses to a scale of the same length, with the same very good points, in
\[
\prod_{\substack{n<\omega\\ i\leq n}}\kappa_{n+1}^{+i}.
\]
\end{proof}

\begin{theorem} \label{badscale}
Suppose $\sigma \in {^\omega \omega}$ and, for all $n<\omega$, $\sigma(n) \geq n$. Then, in $V[G*H]$, there is a bad scale of length $\mu$ in 
\[
\prod_{n<\omega} \kappa_n^{+\sigma(n)}.
\]
\end{theorem}

\begin{proof}
Since, in $V[G]$, $\kappa$ is supercompact, there is a scale $\overrightarrow{g} = \langle g_\alpha \mid \alpha < \mu \rangle$ in 
\[
\prod_{n<\omega}\kappa^{+\sigma(n)}
\]
with stationarily many bad points of cofinality $<\kappa$ (see \cite{canonicalstructure} for a proof). We have arranged with our preparatory forcing that, for every $n<\omega$, $j^*_n(\kappa) = j^*(\kappa)$. For each $n< \omega$ and each $\eta < \kappa^{+\sigma(n)}$, let $F^n_\eta : \mathcal{P}_\kappa(\kappa^{+n}) \rightarrow \kappa$ be such that $[F^n_\eta]_{U^*_n} = \eta$. We may assume that, for all $x \in \mathcal{P}_\kappa(\kappa^{+n}), \ F^n_\eta(x) < \kappa_x^{+\sigma(n)}$. We now define $\langle f_\alpha \mid \alpha < \mu \rangle$ in 
\[
\prod_{n<\omega} \kappa_n^{+\sigma(n)}
\]
by letting $f_\alpha(n) = F^n_{g_\alpha(n)}(x_n)$.

\begin{claim}
If $\alpha < \alpha' < \kappa^{+\omega+1}$, then $f_\alpha <^* f_{\alpha'}$.
\end{claim}

\begin{proof}
Since $\overrightarrow{g}$ is a scale, there is $n^*$ such that, for all $n\geq n^*$, $g_\alpha(n) < g_{\alpha'}(n)$. Thus, for every $n \geq n^*$, $\{x \in \mathcal{P}_\kappa(\kappa^{+n}) \mid F^n_{g_\alpha(n)}(x) < F^n_{g_{\alpha'}(n)}(x)\} \in U^*_n$, so, by genericity, for large enough $n$, $f_\alpha(n)<f_{\alpha'}(n)$.
\end{proof}

\begin{lemma} \label{bdg2}
In $V[G*H]$, let
\[
h\in \prod_{n<\omega}\kappa_n^{+\sigma(n)}.
\]
Then there is $\langle H_n \mid n<\omega \rangle \in V[G]$ such that $\mathrm{dom}(H_n) = \mathcal{P}_\kappa(\kappa^{+n})$, $H_n(x) < \kappa_x^{+\sigma(n)}$ for all $x \in \mathcal{P}_\kappa(\kappa^{+n})$, and, for large enough $n$, $h(n)<H_n(x_n)$.
\end{lemma}

\begin{proof}
Let $h$ be as in the statement of the lemma, and let $\dot{h} \in V[G]$ be a $\mathbb{Q}$-name for $h$. Let $q\in \mathbb{Q}$. We show that there is $p \leq^* q$ forcing the desired conclusion. As in the proof of Lemma \ref{bdg1}, we assume that $q$ is the trivial condition and that
\[
q \Vdash ``\dot{h}\in \prod_{n<\omega}\dot{\kappa}_n^{+\sigma(n)}".
\]

Work in $V[G]$. If $s$ is a lower part of length $n+1$ with maximum element $x^s_n$, then $s^\frown \mathbbm{1} \Vdash ``\dot{h}(n)<\kappa_{x^s_n}^{+\sigma(n)} < \kappa"$. Thus, by the Prikry property and the $\kappa$-completeness of the measures, there is an upper part $A_s$ and an ordinal $\alpha_s < \kappa_{x^s_n}^{+\sigma(n)}$ such that $s^\frown A_s \Vdash ``\dot{h}(n) = \alpha_s"$. By taking a diagonal intersection, we obtain a condition $q' = \langle B_0, B_1, \ldots \rangle$ such that, for every lower part $s$ of length $n+1$ compatible with $q'$, $s^\frown q' \Vdash ``\dot{h}(n) = \alpha_s"$.

For $x \in \mathcal{P}_\kappa(\kappa^{+n})$, let $H_n(x) = \sup(\{\alpha_s \mid s$ is a lower part of length $n+1$ with top element $x \})$. Note that, if $m<n, y\in \mathcal{P}_\kappa(\kappa^{+m})$, and $y\prec x$, then $y\subseteq x \cap \kappa^{+m}$. Since $|x \cap \kappa^{+m}| = \kappa_x^{+m}$, there are fewer than $\kappa_x^{+n} \leq \kappa_x^{+\sigma(n)}$-many lower parts of length $n+1$ with top element $x$, so $H_n(x) < \kappa_x^{+\sigma(n)}$. Moreover, it is clear that, for every $n<\omega$, $q'\Vdash ``\dot{h}(n) < H_n(\dot{x}_n)"$.
\end{proof}

\begin{claim}
$\overrightarrow{f} = \langle f_\alpha \mid \alpha < \mu \rangle$ is cofinal in
\[
\prod_{n<\omega} \kappa_n^{+\sigma(n)}
\]
\end{claim}

\begin{proof}
Let
\[
h\in \prod_{n<\omega} \kappa_n^{+\sigma(n)}
\]
and let $\langle H_n \mid n<\omega \rangle \in V[G]$ be as given by the previous lemma. For each $n<\omega$, $[H_n]_{U^*_n} < \kappa^{+\sigma(n)}$, so we can find $\alpha < \mu$ and $n^*<\omega$ such that for all $n\geq n^*$, $[H_n]_{U^*_n} < g_\alpha(n)$. Then, for all $n \geq n^*$, $\{x \in \mathcal{P}_\kappa(\kappa^{+n}) \mid H_n(x) < F^n_{g_\alpha(n)}(x) \} \in U^*_n$. Thus, by genericity, for large enough $n$, $h(n) < H_n(x_n) < f_\alpha(n)$.
\end{proof}

\begin{claim}
If $\alpha < \mu$ is good for $\overrightarrow{f}$, then it is good for $\overrightarrow{g}$ as well.
\end{claim}

\begin{proof}
Let $\alpha$ be good for $\overrightarrow{f}$. $\omega < \mathrm{cf}(\alpha) < \kappa$, and this is true in $V[G]$ as well. Since every unbounded subset of $\alpha$ in $V[G*H]$ contains an unbounded subset in $V[G]$, we can choose $A\in V[G]$ unbounded in $\alpha$ and $n^* < \omega$ witnessing that $\alpha$ is good for $\overrightarrow{f}$. Moreover, we may assume that $\mathrm{otp}(A) = \mathrm{cf}(A)$. Let $q = \langle x_0, x_1, \ldots , x_{n-1}, A_n, A_{n+1}, \ldots \rangle$ force that $A$ and $n^*$ witness the goodness of $\alpha$. It must be the case that for every $m \geq n,n^*$, $\{x \in \mathcal{P}_\kappa(\kappa^{+m}) \mid \langle F^m_{g_\beta(m)}(x) \mid \beta \in A \rangle$ is strictly increasing$\} \in U^*_m$, since otherwise we could find $p \leq q$ forcing that $\langle f_\beta(m) \mid \beta \in A \rangle$ is not strictly increasing. Thus, for all $m \geq n, n^*$ and $\beta, \gamma \in A$ with $\beta < \gamma$, $g_\beta(m) = [F^m_{g_\beta(m)}]_{U^*_m} < [F^m_{g_\gamma(m)}]_{U^*_m} = g_\gamma(m)$. Thus, $A$ witnesses that $\alpha$ is good for $\overrightarrow{g}$.
\end{proof}
We know that, in $V[G]$, there is a stationary set of $\alpha < \mu$ with $\omega < \mathrm{\alpha} < \kappa$ such that $\alpha$ is bad for $\overrightarrow{g}$. Since $\mathbb{Q}$ has the $\mu$-c.c., this set remains stationary in $V[G*H]$. Thus, $\overrightarrow{f}$ is a bad scale in $V[G*H]$.
\end{proof}

In \cite{gitiksharon}, Gitik and Sharon show that, in $V[G*H]$, there is a very good scale in $\prod \kappa_n^{+\omega+1}$ of length $\mu$ and a scale in $\prod \kappa_n^{+\omega+2}$ of length $\mu^+$ such that every $\alpha < \mu^+$ with $\omega < \mathrm{cf}(\alpha) < \kappa$ is very good. We now show that, above this, there is no essentially new behavior.

\begin{theorem}
In $V[G]$, let $\sigma: \omega \rightarrow \kappa$ be such that, for all $n<\omega$, $\sigma(n) \geq \omega+1$. Then, in $V[G*H]$, there is a scale of length $\mu^+$ in $\prod \kappa_n^{+\sigma(n)+1}$ such that every $\alpha<\mu^+$ with $\omega < \mathrm{cf}(\alpha) < \kappa$ is very good.
\end{theorem}

\begin{proof}
First note that, by genericity, for large enough $n$, $\kappa_n^{+\sigma(n)+1} < \kappa_{n+1}$. Thus, we may assume without loss of generality that this is true for all $n<\omega$. Work in $V[G]$. For $n<\omega$, let $\eta_n = (\kappa^{+\sigma(n)+1})^M$. $\eta_n < j(\kappa)$ and $|j(\kappa)| = \mu^+$, so, since $M$ is closed under $\mu$-sequences of ordinals, $|\eta_n| = \mathrm{cf}(\eta_n) = \mu^+$. Let $\langle \alpha^n_\zeta \mid \zeta < \mu^+ \rangle$ be increasing, continuous, and cofinal in $\eta_n$.

Recall, letting $n=0$ in Lemma \ref{preplem}, that for all $\alpha < j(\kappa)$ (so certainly for all $\alpha < \eta_n$), there is $g_\alpha: \kappa \rightarrow \kappa$ such that $j^*(g_\alpha)(\kappa) = \alpha$. Now, moving to $V[G*H]$, define $\overrightarrow{f} = \langle f_\zeta \mid \zeta < \mu^+ \rangle$ by letting $f_\zeta(n) = g_{\alpha^n_\zeta}(\kappa_n)$. The proofs of the following claims are only minor modifications of the proofs of the analogous claims from Theorem \ref{vgthm} and are thus omitted.
\begin{claim}
For all $\zeta < \mu^+$, for all large enough $n<\omega$, $f_\zeta(n) < \kappa_n^{+\sigma(n)+1}$.
\end{claim}

\begin{claim}
In $V[G*H]$, let \[h \in \prod_{n<\omega} \kappa_n^{+\sigma(n)+1}.\] Then there is $\langle H_n \mid n<\omega \rangle \in V[G]$ such that $\mathrm{dom}(H_n) = \mathcal{P}_\kappa(\kappa^{+n})$, $H_n(x) < \kappa_x^{+\sigma(n)}$ for all $x \in \mathcal{P}_\kappa(\kappa^{+n})$, and, for large enough $n$, $h(n) < H_n(x_n)$.
\end{claim}

\begin{claim}
$\overrightarrow{f}$ is a scale in $\prod \kappa_n^{+\sigma(n)+1}$.
\end{claim}

\begin{claim}
If $\alpha < \mu^+$ and $\omega < \mathrm{cf}(\alpha) < \kappa$, then $\alpha$ is very good for $\overrightarrow{f}$.
\end{claim}
\end{proof}

We now take a step back momentarily to survey the landscape. Things become a bit clearer if, in $V[G*H]$, we force with $\mathrm{Coll}(\mu, \mu^+)$, producing a generic object $I$. Since this forcing is so highly closed, all relevant scales in $V[G*H]$ remain scales in $V[G*H*I]$, and the goodness or badness of points of uncountable cofinality is preserved. The only thing that is changed is that, in $V[G*H*I]$, all relevant scales have length $\mu = \kappa^+$. Moreover, we have a very detailed picture of which scales are good and which are bad. Let $\sigma \in V[G]$ with $\sigma: \omega \rightarrow \kappa$ and, for all $n<\omega$, either $\sigma(n) = 0$ or $\sigma(n)$ is a successor ordinal, and consider a scale $\overrightarrow{f}$ of length $\mu$ in $\prod \kappa_n^{+\sigma(n)}$. First consider the case $\sigma: \omega \rightarrow \omega$. If, for large enough $n$, $\sigma(n)<n$, then $\overrightarrow{f}$ is a good scale, and in fact there is a very good scale in the same product. On the other hand, if $\sigma(n) \geq n$ for infinitely many $n$, then $\overrightarrow{f}$ is bad. Thus, the diagonal sequence $\langle \kappa_n^{+n} \mid n<\omega \rangle$ is a dividing line between goodness and badness in the finite successors of the $\kappa_n$'s. If, alternatively, $\sigma(n) > \omega$ for all sufficiently large $n$, then $\overrightarrow{f}$ is once again a good scale, and there is a very good scale in the same product.

\section{Very Weak Square in the Gitik-Sharon Model}

We take a brief moment to note that, though $AP_\kappa$ necessarily fails in the forcing extension by $\mathbb{Q}$, the weaker Very Weak Square principle may hold. We first recall the following definition from \cite{foremanmagidor}.

\begin{definition}
Let $\lambda$ be a singular cardinal. A {\it Very Weak Square sequence} at $\lambda$ is a sequence $\langle C_\alpha \mid \alpha < \lambda^+ \rangle$ such that, for a club of $\alpha < \lambda^+$, 
\begin{itemize}
\item{$C_\alpha$ is an unbounded subset of $\alpha$.}
\item{For all bounded $x \in [C_\alpha]^{<\omega_1}$, there is $\beta < \alpha$ such that $x = C_\beta$.}
\end{itemize}
\end{definition}
Note that we may assume in the above definition that, for the relevant club of $\alpha < \lambda^+$, $\mathrm{otp}(C_\alpha) = \mathrm{cf}(\alpha)$.

The existence of a Very Weak Square sequence at $\lambda$ follows from $AP_\lambda$, but the converse is not true. In fact, in \cite{foremanmagidor}, Foreman and Magidor prove that the existence of a Very Weak Square sequence at every singular cardinal is consistent with the existence of a supercompact cardinal. Also, note that a Very Weak Square sequence at $\lambda$ is preserved by any countably-closed forcing which also preserves $\lambda$ and $\lambda^+$. In particular, our preparation forcing $\mathbb{P}$ preserves Very Weak Square sequences. Thus, we may assume that, prior to forcing with $\mathbb{Q}$, there is a Very Weak Square sequence at $\kappa^{+\omega}$. 

Let $V$ denote the model over which we will force with $\mathbb{Q}$, and suppose that $\overrightarrow{C} = \langle C_\alpha \mid \alpha < \mu \rangle$ is a Very Weak Square sequence in $V$, where $\mu = \kappa^{+\omega+1}$. Assume additionally that there is a club $E \subseteq \mu$ such that for all $\alpha \in E$, $\mathrm{otp}(C_\alpha) = \mathrm{cf}(\alpha)$ and for all bounded $x \in [C_\alpha]^{<\omega_1}$, there is $\beta < \alpha$ such that $x = C_\beta$.  Let $G$ by $\mathbb{Q}$-generic over $V$. In $V[G]$, form $\overrightarrow{D} = \langle D_\alpha \mid \alpha < \mu \rangle$ as follows.
\begin{itemize}
\item{If $\alpha \not\in E$ or $\alpha \in E$ and $\mathrm{cf}^V(\alpha) < \kappa$, let $D_\alpha = C_\alpha$.}
\item{If $\alpha \in E$ and $\mathrm{cf}^V(\alpha) \geq \kappa$, let $D_\alpha \subseteq C_\alpha$ be an $\omega$-sequence cofinal in $\alpha$.}
\end{itemize}

Now, using the fact that forcing with $\mathbb{Q}$ does not add any bounded sequences of $\kappa$, it is easy to verify that $\overrightarrow{D}$ is a Very Weak Square sequence at $\kappa$.

\section{Classifying Bad Points}

We now turn our attention to Cummings and Foreman's second question from \cite{cummingsforeman}. We first recall some relevant definitions and the Trichotomy Theorem, due to Shelah \cite{shelah}.

\begin{definition}
Let $X$ be a set, let $I$ be an ideal on $X$, and let $f,g: X \rightarrow ON$. Then $f<_I g$ if $\{x\in X \mid g(x) \leq f(x) \} \in I$. $\leq_I$, $=_I$, $>_I$, and $\geq_I$ are defined analogously. If $D$ is the dual filter to $I$, then $<_D$ is the same as $<_I$.
\end{definition}

Thus, if $X$ is a set of ordinals and $I$ is the ideal of bounded subsets of $X$, then $<_I$ is the same as $<^*$.

\begin{definition}
Let $I$ be an ideal on $X$, $\beta$ an ordinal, and $\overrightarrow{f} = \langle f_\alpha \mid \alpha < \beta \rangle$ a $<_I$-increasing sequence of functions in ${^X}ON$. $g\in {^X}ON$ is an {\it exact upper bound} (or {\it eub}) for $\overrightarrow{f}$ if the following hold:
\begin{enumerate}
\item{For all $\alpha < \beta$, $f_\alpha <_I g$.}
\item{For all $h\in {^X}ON$ such that $h <_I g$, there is $\alpha < \beta$ such that $h <_I f_\alpha$.}
\end{enumerate}
\end{definition}

We note that, easily, if $\overrightarrow{f}$ is a $<_I$-increasing sequence of functions and $g$ and $h$ are both eubs for $\overrightarrow{f}$, then $g=_Ih$. The following is a standard alternate characterization of good points in scales.

\begin{proposition} \label{goodprop}
Let $\kappa$ be singular, let $A \subseteq \kappa$ be a cofinal set of regular cardinals of order type $\mathrm{cf}(\kappa)$, and let $\overrightarrow{f} = \langle f_\alpha \mid \alpha < \mu \rangle$ be a scale in $\prod A$. Let $\beta < \mu$ be such that $\mathrm{cf}(\kappa) < \mathrm{cf}(\beta) < \kappa$. Then the following are equivalent:
\begin{enumerate}
\item{$\beta$ is good for $\overrightarrow{f}$.}
\item{$\langle f_\alpha \mid \alpha < \beta \rangle$ has an eub, $g$, such that, for all $i\in A$, $\mathrm{cf}(g(i)) = \mathrm{cf}(\beta)$.}
\item{There is a $<$-increasing sequence of functions $\langle h_\xi \mid \xi < \mathrm{cf}(\beta) \rangle$ that is cofinally interleaved with $\overrightarrow{f}\restriction \beta$, i.e. for every $\alpha < \beta$ there is $\xi < \mathrm{cf}(\beta)$ such that $f_\alpha <^* h_\xi$ and, for every $\xi < \mathrm{cf}(\beta)$, there is $\alpha < \beta$ such that $h_\xi <^* f_\alpha$. In this case, the function $i\mapsto \sup(\{h_\xi(i) \mid \xi < \mathrm{cf}(\beta)\})$ is an eub for $\overrightarrow{f}\restriction \beta$.}
\end{enumerate}
\end{proposition}

\begin{theorem} (Trichotomy)
Suppose $I$ is an ideal on $X$, $|X|^+ < \lambda = \mathrm{cf}(\lambda)$, and $\langle f_\alpha \mid \alpha < \lambda \rangle$ is a $<_I$-increasing sequence of functions in ${^X}ON$. Then one of the following holds:
\begin{enumerate}
\item{(Good) $\langle f_\alpha \mid \alpha < \lambda \rangle$ has an eub, $g$, such that, for all $x\in X$, $\mathrm{cf}(g(x)) > |X|$.}
\item{(Bad) There is an ultrafilter $U$ on $X$ extending the dual filter to $I$ and a sequence $\langle S_x \mid x\in X \rangle$ such that $|S_x| \leq |X|$ for all $x\in X$ and, for all $\alpha < \lambda$, there are $h\in \prod_{x\in X} S_x$ and $\beta < \lambda$ such that $f_\alpha <_U h <_U f_\beta$.}
\item{(Ugly) There is a function $h \in {^X}ON$ such that the sequence of sets $\langle \{x \mid f_\alpha(x) < h(x) \} \mid \alpha < \lambda \rangle$ does not stabilize modulo $I$.}
\end{enumerate}
\end{theorem}

We note that the above terminology is slightly misleading. For $\beta$ to be a good point in a scale, for example, requires more than $\langle f_\alpha \mid \alpha < \beta \rangle$ falling into the Good case of the Trichotomy Theorem. It also requires that the eub have uniform cofinality equal to $\mathrm{cf}(\beta)$. Also, $\beta$ being a bad point in a scale does not imply that $\langle f_\alpha \mid \alpha < \beta \rangle$ falls into the Bad case of the Trichotomy Theorem.

We now answer Cummings and Foreman's question asking into which case of the Trichotomy Theorem the bad points in the Gitik-Sharon model fall. We also answer the analogous question for some other models in which bad scales exist, showing that, in the standard models in which bad scales exist at relatively small cardinals, there is considerable diversity of behavior at the bad points. We first recall the following fact from \cite{canonicalstructure}.

\begin{fact} \label{scbadness}
Suppose $\kappa$ is a supercompact cardinal, $\sigma \in {^\omega}\omega$ is a function such that, for all $n<\omega$, $\sigma(n) \geq n$, and $\langle f_\alpha \mid \alpha < \kappa^{+\omega+1} \rangle$ is a scale in $\prod_{n<\omega} \kappa^{+\sigma(n)}$. Then there is an inaccessible cardinal $\delta < \kappa$ such that, for stationarily many $\beta \in \kappa^{+\omega+1} \cap \mathrm{cof}(\delta^{+\omega+1})$, $\langle f_\alpha \mid \alpha < \beta \rangle$ has an eub, $g$, such that, for all $n<\omega$, $\mathrm{cf}(g(n)) = \delta^{+\sigma(n)}$.
\end{fact}

The content of the next theorem is that these eubs of non-uniform cofinality get transferred down to the bad scales defined in extensions by diagonal supercompact Prikry forcing. Note that the proof of the existence of bad scales in Theorem \ref{badscale} did not rely on our preparatory forcing $\mathbb{P}$, so we dispense with it here.

\begin{theorem}
Let $\kappa$ be supercompact, let $\mu = \kappa^{+\omega+1}$, and let $\sigma \in {^\omega}\omega$ be such that, for all $n<\omega$, $\sigma(n) \geq n$. Let $\mathbb{Q}$ be diagonal supercompact Prikry forcing at $\kappa$ defined from $\langle U_n \mid n<\omega \rangle$, where $U_n$ is a measure on $\mathcal{P}_\kappa(\kappa^{+n})$. In $V^\mathbb{Q}$, let $\overrightarrow{f}$ be the bad scale in $\prod_{n<\omega} \kappa_n^{+\sigma(n)}$ defined as in Theorem \ref{badscale}. Then there is an inaccessible cardinal $\delta < \kappa$ such that, for stationarily many $\beta \in \mu \cap \mathrm{cof}(\delta^{+\omega+1})$, $\langle f_\alpha \mid \alpha < \beta \rangle$ has an eub, $g$, such that, for all $n<\omega$, $\mathrm{cf}(g(n)) = \delta^{+\sigma(n)}$.
\end{theorem}

\begin{proof}

Let $G$ be $\mathbb{Q}$-generic over $V$, and let $\langle x_n \mid x<\omega \rangle$ be the associated generic sequence. $\overrightarrow{f} = \langle f_\alpha \mid \alpha < \mu \rangle$ is the bad scale in $\prod_{n<\omega} \kappa_n^{+\sigma(n)}$ defined as in the proof of Theorem \ref{badscale}, i.e. $f_\alpha(n) = F^n_{g_\alpha(n)}(x_n)$, where $\overrightarrow{g} = \langle g_\alpha \mid \alpha < \mu \rangle \in V$ is a scale in $\prod_{n<\omega} \kappa^{+\sigma(n)}$ and, for each $n<\omega$ and $\eta < \kappa^{+\sigma(n)}$, $F^n_\eta : \mathcal{P}_\kappa(\kappa^{+n}) \rightarrow \kappa$ is such that $[F^n_\eta]_{U_n} = \eta$.

By Fact \ref{scbadness}, there is an inaccessible $\delta < \kappa$ and a stationary $S\subseteq \mu \cap \mathrm{cof}(\delta^{+\omega+1})$ such that, for all $\beta \in S$, $\langle g_\alpha \mid \alpha < \beta \rangle$ has an eub, $g$, such that, for all $n<\omega$, $\mathrm{cf}(g(n)) = \delta^{+\sigma(n)}$. Without loss of generality, $\overrightarrow{g}$ is a continuous scale and this eub is in fact $g_\beta$.

Since $\mathbb{Q}$ has the $\mu$-c.c. and preserves the inaccessibility of $\delta$ and the regularity of $\delta^{+\omega+1}$, $S$ remains a stationary subset of $\mu \cap \mathrm{cof}(\delta^{+\omega+1})$ in $V[G]$. Thus, we will be done if we show that, for all $\beta \in S$, $\langle f_\alpha \mid \alpha < \beta \rangle$ has an eub $g$ such that, for all $n<\omega$, $\mathrm{cf}(g(n)) = \delta^{+\sigma(n)}$. In fact, we claim that this eub is, up to finite adjustments, $f_\beta$.

\begin{claim}
Let $\beta \in S$. Then, for sufficiently large $n<\omega$, $\mathrm{cf}(f_\beta(n)) = \delta^{+\sigma(n)}$
\end{claim}

\begin{proof}
For all $n$, $\mathrm{cf}(g_\beta(n)) = \delta^{+\sigma(n)}$. Thus, since $\delta^{+\sigma(n)} < \kappa$ and $[F^n_{g(\beta)}]_{U_n} = g(\beta)$, we know that, for all $n<\omega$ and almost all $x\in X_n$, $\mathrm{cf}(F^n_{g_\beta(n)}(x)) = \delta^{+\sigma(n)}$. Therefore, by genericity of $\langle x_n \mid n<\omega \rangle$, $\mathrm{cf}(f_\beta(n)) = \mathrm{cf}(F^n_{g_\beta(n)}(x_n)) = \delta^{+\sigma(n)}$ for all sufficiently large $n<\omega$.
\end{proof}

\begin{claim}
Let $\beta \in S$. Then $f_\beta$ is an eub for $\langle f_\alpha \mid \alpha < \beta \rangle$.
\end{claim}

\begin{proof}
In $V[G]$, fix $h \in \prod_{n<\omega} \kappa_n^{+\sigma(n)}$ such that $h < f_\beta$. We want to find $\alpha < \beta$ such that $h <^* f_\alpha$. We first define some auxiliary functions.

In $V$, for each $n<\omega$, let $\langle \eta^n_\xi \mid \xi < \delta^{+\sigma(n)} \rangle$ be an increasing sequence of ordinals cofinal in $g_\beta(n)$. In $V[G]$, for $\tau \in \prod_{n<\omega} \delta^{+\sigma(n)}$, define $f_\tau$ by letting, for all $n<\omega$, $f_\tau(n) = F^n_{\eta^n_{\tau(n)}}(x_n)$.

Since, for all $n<\omega$, $\langle \eta^n_\xi \mid \xi < \delta^{+\sigma(n)} \rangle$ is increasing and cofinal in $g_\beta(n)$ and since $\delta^{+\sigma(n)} < \kappa$, we have, by the $\kappa$-completeness of the measures, that, for all $n<\omega$, the set of $x\in X_n$ such that $\langle F^n_{\eta^n_\xi}(x)\mid \xi < \delta^{+\sigma(n)} \rangle$ is increasing and cofinal in $F^n_{g_\beta(n)}(x)$ is in $U_n$. Thus, for sufficiently large $n<\omega$, $\langle F^n_{\eta^n_\xi}(x_n) \mid \xi < \delta^{+\sigma(n)} \rangle$ is increasing and cofinal in $F^n_{g_\beta(n)}(x_n) = f_\beta(n)$. Thus, we can find $\tau \in \prod_{n<\omega} \delta^{+\sigma(n)}$ such that, for large enough $n<\omega$, $h(n) < F^n_{\eta^n_{\tau(n)}}(x_n)$, i.e. $h<^*f_\tau$.

Since $\mathbb{Q}$ does not add any bounded subsets of $\kappa$, we actually have $\tau \in V$. For all $n<\omega$, $\eta^n_{\tau(n)} < g_\beta(n)$, so, since $g_\beta$ is an eub for $\langle g_\alpha \mid \alpha < \beta \rangle$, there is $\alpha < \beta$ such that, for large enough $n<\omega$, $\eta^n_{\tau(n)} < g_\alpha(n)$. Thus, we know by genericity that, again for sufficiently large $n<\omega$, $F^n_{\eta^n_{\tau(n)}}(x_n) < F^n_{g_\alpha(n)}(x_n)$, i.e. $f_\tau <^* f_\alpha$. So $h<^*f_\tau <^* f_\alpha$, and we have have shown that $f_\beta$ is in fact an eub for $\langle f_\alpha \mid \alpha < \beta \rangle$ and hence proven the theorem.
\end{proof}
\end{proof}

Thus, the bad scales in the Gitik-Sharon model have stationarily many points which lie in the Good case of the Trichotomy Theorem but are nonetheless bad, since the eubs at these points have non-uniform cofinality. 

We now turn our attention to a bad scale isolated by Cummings, Foreman, and Magidor in \cite{canonicalstructure}. Let $\kappa$ be a supercompact cardinal and let $\overrightarrow{f} = \langle f_\alpha \mid \alpha < \kappa^{+\omega+1} \rangle$ be a continuous scale in $\prod_{n<\omega} \kappa^{+n}$. Let $\delta$ be as given in Fact \ref{scbadness} and let $G_0 \times G_1$ be $\mathrm{Coll}(\omega, \delta^{+\omega}) \times \mathrm{Coll}(\delta^{+\omega+2}, <\kappa)$-generic over $V$. Cummings, Foreman, and Magidor show that, in $V[G_0 \times G_1]$, $\overrightarrow{f}$ remains a scale, now living in $\prod_{n<\omega} \aleph_{n+3}$, and that the stationary set $S$ of bad points in $V$ of cofinality $\delta^{+\omega+1}$ is a stationary set of bad points in $V[G_0 \times G_1]$ of cofinality $\omega_1$. The Trichotomy Theorem does not apply to points of cofinality $\omega_1$ in increasing sequences of countable reduced products (see \cite{kojman} for a counterexample), but we show that the points in $S$ nonetheless fall into the Ugly case of the Trichotomy Theorem.

\begin{theorem}
In $V[G_0 \times G_1]$, if $\beta \in S$, then there is $h<f_\beta$ such that the sequence of sets $\langle \{n \mid f_\alpha(n) < h(n) \} \mid \alpha < \beta \rangle$ does not stabilize modulo bounded sets.
\end{theorem}

\begin{proof}
Let $\beta \in S$. We must produce an $h$ such that, for every $\alpha < \beta$, there is $\alpha'<\beta$ such that $\{n \mid f_\alpha(n) < h(n) < f_{\alpha'}(n) \}$ is infinite. We will actually show that such an $h$ exists in $V[G_0]$. First, for concreteness, we remark that we are thinking of conditions in $\mathrm{Coll}(\omega, \delta^{+\omega})$ as finite partial functions from $\omega$ into $\delta^{+\omega}$. Let $g$ be the generic surjection from $\omega$ onto $\delta^{+\omega}$ added by $G_0$. In $V$, we know that, for every $n<\omega$, $\mathrm{cf}(f_\beta(n)) = \delta^{+n}$. For each $n<\omega$, let $\langle \eta^n_\xi \mid \xi < \delta^{+n} \rangle$ be increasing and cofinal in $f_\beta(n)$. In $V[G_0]$, define $h$ as follows: if $n$ is such that $g(n)<\delta^{+n}$, then let $h(n) = \eta^n_{g(n)}$. If $g(n) > \delta^{+n}$, let $h(n)=0$. 

Clearly, $h < f_\beta$. Let $\alpha < \beta$. In $V$, there is $\tau \in \prod \delta^{+n}$ such that, for sufficiently large $n$, $f_\alpha(n) < \eta^n_{\tau(n)}$. Also, since $f_\beta$ is an eub, there is $\alpha' < \beta$ such that, again, for sufficiently large $n$, $\eta^n_{\tau(n)} < f_{\alpha'}(n)$. An easy density argument shows that, for infinitely many $n<\omega$, $g(n) = \tau(n)$ and thus $h(n) = \eta^n_{\tau(n)}$. Therefore, for infinitely many $n<\omega$, $f_\alpha(n) < h(n) < f_{\alpha'}(n)$.
\end{proof}

We end this section by briefly remarking on two other models in which bad scales exist. A result of Magidor \cite{cummingsmagidor} shows that, if Martin's Maximum holds, then any scale of length $\aleph_{\omega+1}$ in $\prod_A \aleph_n$, where $A\subseteq \omega$, is bad. Foreman and Magidor, in \cite{foremanmagidor}, show that the same conclusion follows from the Chang's Conjecture $(\aleph_{\omega+1}, \aleph_\omega) \twoheadrightarrow (\aleph_1, \aleph_0)$. The proofs of these results immediately yield that, in both cases, such a scale has stationarily many points of cofinality $\omega_1$ that fall into the Bad case of the Trichotomy Theorem.

\section{Down to $\aleph_{\omega^2}$}

In \cite{gitiksharon}, Gitik and Sharon show how to arrange so that $\kappa$, which is supercompact in $V$, becomes $\aleph_{\omega^2}$ in the forcing extension, SCH and approachability both fail at $\aleph_{\omega^2}$, and there is $A\subseteq \aleph_{\omega^2}$ that carries a very good scale. We start this section by reviewing their construction, being slightly more careful with our preparation of the ground model so that our results from Section 3 carry down.

In $V$, let $j:V\rightarrow M$ be the elementary embedding derived from $U$, a supercompactness measure on $\mathcal{P}_\kappa(\kappa^{+\omega+1})$. Let $\mathbb{P}$ be the backward Easton-support iteration from Section 3, and let $j^*:V[G]\rightarrow M[G*H*I]$ be as in Lemma \ref{preplem}. Let $U^*$ be the measure on $\mathcal{P}_\kappa(\kappa^{+\omega+1})$ derived from $j^*$ and, for $n<\omega$, let $U_n^*$ be the projection of $U^*$ onto $\mathcal{P}_\kappa(\kappa^{+n})$. Let $i_n^*:V[G] \rightarrow N_n$ be the elementary embedding derived from the ultrapower of $V[G]$ by $U^*_n$ and let $k_n:N_n \rightarrow M[G*H*I]$ be the factor map. The functions $\{g^n_\beta \mid \beta < j(\kappa^{+n})\}$ witness that $i_n^*(\kappa^{+n})=j(\kappa^{+n})$, so $\mathrm{crit}(k_n) > j(\kappa^{+n})$.

As before, we can find, in $V[G]$, an $H^*$ that is $\mathrm{Coll}(\kappa^{+\omega+2}, <j(\kappa))^{M[G*H*I]}$-generic over $M[G*H*I]$. For all $n<\omega$, since $\mathrm{crit}(k_n) > j(\kappa)$ and $\mathrm{Coll}(\kappa^{+\omega+2}, <i_n^*(\kappa))^{N_n}$ has the $i_n^*(\kappa)$-c.c., the filter generated by $k_n^{-1}[H^*]$, which we will call $H_n$, is $\mathrm{Coll}(\kappa^{+\omega+2}, <i_n^*(\kappa))^{N_n}$-generic over $N_n$.

In $V[G]$, let $\delta < \kappa$ be an inaccessible cardinal such that, for every $\sigma \in {^\omega}\omega$ such that $\sigma(n) \geq n$ for all $n<\omega$, there is a scale $\overrightarrow{f}$ in $\prod_{n<\omega} \kappa^{+\sigma(n)}$ of length $\kappa^{+\omega+1}$ such that there are stationarily many $\beta < \kappa^{+\omega+1}$ of cofinality $\delta^{+\omega+1}$ such that $\overrightarrow{f}\restriction \beta$ has an eub, $g$, such that, for all $n<\omega$, $\mathrm{cf}(g(n)) = \delta^{+\sigma(n)}$. Such a $\delta$ exists by the proof of Fact \ref{scbadness}, which can be found in \cite{canonicalstructure}. We now define a version of the diagonal supercompact Prikry forcing with interleaved collapses. Conditions in this poset, which we again call $\mathbb{Q}$, are of the form 
\[
q=\langle c^q, x_0^q, f_0^q, x_1^q, f_1^q, \ldots ,x_{n-1}^q, f_{n-1}^q, A_n^q, F_n^q, A_{n+1}^q, F_{n+1}^q, \ldots \rangle
\]
such that the following conditions hold:
\begin{enumerate}
	\item{$\langle x_0^q, x_1^q, \ldots ,x_{n-1}^q, A_n^q, A_{n+1}^q \ldots \rangle$ is a condition in the diagonal supercompact Prikry forcing defined using the $U^*_n$'s.}
	\item{If $n=0$, $c^q \in \mathrm{Coll}(\omega, <\delta) \times \mathrm{Coll}(\delta^{+\omega+2}, <\kappa)$.}
	\item{If $n\geq 1$, $c^q \in \mathrm{Coll}(\omega, <\delta) \times \mathrm{Coll}(\delta^{+\omega+2}, <\kappa_{x_0^q})$.}
	\item{For all $i<n-1$, $f^q_i \in \mathrm{Coll}(\kappa_{x^q_i}^{+\omega+2}, <\kappa_{x^q_{i+1}})$.}
	\item{$f^q_{n-1} \in \mathrm{Coll}(\kappa_{x^q_{n-1}}^{+\omega+2}, <\kappa)$.}
	\item{For all $\ell \geq n$, $F_\ell^q$ is a function with domain $A_\ell^q$ such that $F_\ell^q(x)\in \mathrm{Coll}(\kappa_x^{+\omega+2}, < \kappa)$ for all $x\in A_\ell^q$ and $i_\ell^*(F_\ell^q)(i_\ell``\kappa^{+\ell}) \in H_\ell$.}
\end{enumerate}
As before, $n$ is the length of $q$, denoted $\mathrm{lh}(q)$. If $p,q \in \mathbb{Q}$, where $\mathrm{lh}(p)=n$ and $\mathrm{lh}(q)=m$, then $p\leq q$ if and only if:
\begin{enumerate}
	\item{$n\geq m$ and $\langle x_0^p, \ldots, x_{n-1}^p, A_n^p, \ldots \rangle \leq \langle x_0^q, \ldots, x_{m-1}^q, A_m^q, \ldots \rangle$ in the standard supercompact diagonal Prikry poset.}
	\item{$c^p\leq c^q$.}
	\item{For all $i\leq m-1$, $f^p_i \leq f^q_i$.}
	\item{For all $i$ such that $m\leq i \leq n-1$, $f^p_i \leq F^q_i(x^p_i)$.}
	\item{For all $i \geq n$ and all $x\in A^p_i$, $F^p_i(x)\leq F^q_i(x)$.}
\end{enumerate}
$p \leq^* q$ if and only if $p \leq q$ and $\mathrm{lh}(p) = \mathrm{lh}(q)$. If $q=\langle c^q, x_0^q, f_0^q, x_1^q, f_1^q, \ldots ,x_{n-1}^q$, $f_{n-1}^q, A_n^q, F_n^q, A_{n+1}^q, F_{n+1}^q, \ldots \rangle$, then the lower part of $q$, $\langle c^q, x_0^q, f_0^q, x_1^q, f_1^q, \ldots $, $x_{n-1}^q, f_{n-1}^q \rangle$, is denoted $s(q)$.

If $J$ is $\mathbb{Q}$-generic over $V[G]$ and $\langle x_n \mid n<\omega \rangle$ is the associated Prikry sequence, then, for each $n<\omega$, letting $\kappa_n = \kappa_{x_n}$, we have $\kappa_n = (\aleph_{\omega \cdot (n+1)+3})^{V[G*J]}$, $\kappa = (\aleph_{\omega^2})^{V[G*H]}$, and $(\kappa^{+\omega+1})^{V[G]} = (\aleph_{\omega^2+1})^{V[G*J]}$.

The results about scales transfer down in a straightforward manner. We provide some details regarding the proofs of two analogs of Lemmas \ref{bdg1} and \ref{bdg2}, since these are complicated somewhat by the interleaving of collapses.

\begin{lemma}
In $V[G*J]$, let 
\[
h \in \prod_{\substack{n<\omega\\ i\leq n}}\sup(x_{n+1}\cap \kappa^{+i}).
\]
Then there is $\langle h_{n,i} \mid n<\omega, \ i \leq n \rangle \in V[G]$ such that $h_{n,i}:\mathcal{P}_\kappa(\kappa^{+n})\rightarrow \kappa^{+i}$ and, for large enough $n$ and all $i\leq n$, $h(n,i) < h_{n,i}(x_n)$. 
\end{lemma}

\begin{proof}
As before, we may let $\dot{h} \in V[G]$ be a $\mathbb{Q}$-name for $h$, we may assume that 
\[
h \in \prod_{\substack{n<\omega\\ i\leq n}}x_{n+1}\cap \kappa^{+i},
\]
and we may assume we are working in $V[G]$ below the trivial condition. By a routine diagonal construction using the Prikry property, we may find a condition $p \in \mathbb{Q}$ such that $\mathrm{lh}(p) = 0$ and, for every $n < \omega$, for every $q \leq p$ of length $n + 2$, for every $i \leq n$, and for every $\alpha \in x^q_{n+1} \cap \kappa^{+i}$, if $q$ decides the truth value of the statement $``\dot{h}(n,i) = \alpha"$, then $s(q) ^\frown p$ decides the statement as well.

Let $n < \omega$, and let $x \in A^p_{n + 1}$. Let $S(x)$ denote the set of lower parts $s$ of length $n + 1$ such that $x^s_{n} \prec x$. By appeal to the Prikry property, our choice of $p$, and the closure of $\mathrm{Coll}(\kappa^{+ \omega + 2}_x, < \kappa)$, we can find $f_x \leq F^p_{n + 1}(x)$ such that, for every $s \in S(x)$ compatible with $p$, for every $i \leq n$, and for every $\alpha \in x \cap \kappa^{+i}$, there is a lower part $s' \leq s$ such that $s' {^\frown} \langle x, f_x \rangle ^\frown p$ decides the truth value of the statement $``\dot{h}(n,i) = \alpha"$. Find $p' \leq^* p$ such that, for every $n < \omega$ and every $x \in A^{p'}_{n+1}$, $F^{p'}_{n+1}(x) \leq f_x$. Note that $p'$ has the property that, for every $n < \omega$, for every lower part $s$ of length $n+1$ compatible with $p'$, for every $i \leq n$, and for every $x \in A^{p'}_{n+1}$, there is a lower part $s' \leq s$ such that $s' {^\frown} \langle x, F^{p'}_{n+1}(x) \rangle {^\frown} p'$ decides the value of $\dot{h}(n,i)$.

For every lower part $s = \langle c^s, x_0^s, f_0^s, x_1^s, f_1^s, \ldots $, $x_{n-1}^s, f_{n-1}^s \rangle$, let $t(s) = \langle x_0^s, x_1^s, \ldots, x_{n-1}^s \rangle$, i.e. $t(s)$ is the Prikry part of $s$. For all $n < \omega$, all $x \in A^{p'}_{n+1}$, all $i \leq n$, and all Prikry lower parts $t$ of length $n + 1$ below $x$, let $\alpha_{t,x,i}$ be the least $\alpha^* \in x \cap \kappa^{+i}$ such that, for every lower part $s$ with $t(s) = t$, if, for some $\alpha$, $s {^\frown} \langle x, F^{p'}_{n+1}(x) \rangle {^\frown} p' \Vdash ``\dot{h}(n,i) = \alpha"$, then $\alpha < \alpha^*$. Such an $\alpha^*$ must exist because the product of the interleaved Levy collapses appearing in such lower parts has the $\kappa_x$-c.c. 

For each $n$, each $i\leq n$ and each Prikry lower part $t$ of length $n+1$ compatible with $p'$, define a regressive function on $A^{p'}_{n+1}$ which takes $x$ and returns $\alpha_{t,x,i}$. This function is constant, returning value $\beta_{t,i}$, on a measure-one set $B_{t,i} \subseteq A^{p'}_{n+1}$. Let $B_t = \cap_{i\leq n} B_{t,i}$, and let $\langle C_0, C_1, \ldots \rangle$ be the diagonal intersection of the $B_t$'s. Let $q$ be the natural restriction of $p'$ to the measure-one sets $\langle C_0, C_1, \ldots \rangle$. For all $n < \omega$, $i \leq n$, and $x\in C_n$, let $h_{n,i}(x) = \sup(\{\beta_{t,i} \mid t$ is a Prikry lower part of length $n+1$ with top element $x\})$. As before, it is routine to verify that $q$ forces $\langle h_{n,i} \mid n<\omega, \ i \leq n \rangle$ to be as desired.
\end{proof}

The proofs of the analogs of Lemma \ref{bdg2} are mostly similar, exploiting the closure of the interleaved collapsing posets. We provide some details in one specific case, namely the lemma associated with the proof that there is a scale in $\prod_{n<\omega} \kappa_n^{+\omega + 2}$. The proof is a modification of a similar proof in \cite{su}. We would like to thank Spencer Unger for pointing out the difficulties of this case and for directing us to \cite{su}.

\begin{lemma}
In $V[G*J]$, let \[h \in \prod_{n<\omega} \kappa_n^{+\omega+2}.\] Then there is $\langle h_n \mid n<\omega \rangle \in V[G]$ such that $\mathrm{dom}(h_n) = \mathcal{P}_\kappa(\kappa^{+n})$, $h_n(x) < \kappa_x^{+\omega + 2}$ for all $x \in \mathcal{P}_\kappa(\kappa^{+n})$, and, for large enough $n$, $h(n) < h_n(x_n)$.
\end{lemma}

\begin{proof}
As usual, working in $V[G]$, let $\dot{h}$ be a name for $h$ and assume, using the Prikry property, that we are working below a condition $p$ such that $\mathrm{lh}(p) = 0$ and, for all $n < \omega$, for every lower part $s$ of length $n + 1$ compatible with $p$, for every $\alpha < \kappa_{x^s_n}^{+\omega + 2}$, there is a lower part $s' \leq s$  such that $s' {^\frown} p$ decides the truth value of the statement $``\dot{h}(n) = \alpha"$.

For $n < \omega$, $x \in A^p_n$, and $f \in \mathrm{Coll}(\kappa_x^{+\omega + 2}, < \kappa)$, let $h_{n,x}(f) = \sup(\{\alpha \mid$ for some lower part $s$ of length $n + 1$ with top two elements $\langle x, f \rangle$, $s {^\frown} p \Vdash ``\dot{h}(n) = \alpha"\})$. Since the number of lower parts of length $n$ below $x$ is less than $\kappa_x^{+n}$, we have $h_{n,x}(f) < \kappa_x^{+\omega + 2}$. 

For $n < \omega$, let $K_n$ denote the set of functions $F$ on $X_n$ such that, for all $x \in X_n$, $F(x) \in \mathrm{Coll}(\kappa_x^{+\omega + 2}, < \kappa)$ and $i^*_n(F)(i_n``\kappa^{+n}) \in H_n$ (i.e. $K_n$ is the set of functions $F$ that could appear in the $n^{th}$ collapsing coordinate of a condition in $\mathbb{Q}$ of length less than $n$). For $n < \omega$, let $\alpha_n = \sup(\{[x \mapsto h_{n,x}(F(x))]_{U_n^*} \mid F \in K_n\})$.

We claim that $\alpha_n < \kappa^{+\omega + 2}$. To see this, we define, for all $x \in A^p_n$ and every lower part $s$ of length $n$ below $x$, $h_{n,s,x}(f) = \alpha$ if $s {^\frown} \langle x, f \rangle {^\frown} p \Vdash ``\dot{h}(n) = \alpha"$ and $h_{n,s,x}(f) = 0$ if there is no such $\alpha$. If $F_0, F_1 \in K_n$, then $F_0(x)$ and $F_1(x)$ are compatible for almost every $x$, so, for every lower part $s$ of length $n$, $\{[x \mapsto h_{n,s,x}(F(x))]_{U_n^*} \mid F \in K_n\}$ has at most one non-zero element, which must be less than $\kappa^{+\omega + 2}$. Let $S^* = \{[x \mapsto \bar{s}(x)]_{U^*_n} \mid$ for all $x \in X_n, \bar{s}(x) \in S(x)\}$. Since, for each $x \in X_n$, $|S(x)| < \kappa_x^{+n}$, we have $|S^*| < \kappa^{+n}$. For functions $\bar{s}$ such that $[\bar{s}]_{U^*_n} \in S^*$, let $\alpha_{n, \bar{s}} = \sup(\{[x \mapsto h_{n, \bar{s}(x), x}(F(x))]_{U_n^*} \mid F \in K_n\}$. By the above comments, $\alpha_{n, \bar{s}} < \kappa^{+\omega + 2}$. Finally, note that $\alpha_n = \sup(\{\alpha_{n, \bar{s}} \mid [\bar{s}]_{U^*_n} \in S^* \})$. Since $|S^*| < \kappa^{+n}$, we have $\alpha_n < \kappa^{+\omega + 2}$.

For all $\beta < \kappa^{+\omega + 2}$, fix a function $g_\beta:\kappa \rightarrow \kappa$ and such that $j^*(g_\beta)(\kappa) = \beta$ such that, for every $\gamma < \kappa$, $g(\gamma) < \gamma^{+\omega + 2}$. For $n < \omega$, define $h_n$ by letting $h_n(x) = g_{\alpha_n}(\kappa_x)$. It is routine to check that $p$ forces $\langle h_n \mid n<\omega \rangle$ to be as desired.
\end{proof}

Now we can prove as before that, in $V[G*J]$, we have the following situation:
\begin{itemize}
\item{$\prod_{n<\omega}\aleph_{\omega \cdot n+1}$ carries a very good scale of length $\aleph_{\omega^2+1}$.}
\item{$\prod_{n<\omega}\aleph_{\omega \cdot n+2}$ carries a scale of length $\aleph_{\omega^2+2}$ such that, for every $\beta < \aleph_{\omega^2+2}$ such that $\omega_1 \leq \mathrm{cf}(\beta) < \aleph_{\omega^2}$, $\beta$ is very good for the scale.}
\item{If $\sigma \in {^\omega}\omega$ is such that, for every $n<\omega$, $\sigma(n) < n$, then $\prod_{n<\omega}\aleph_{\omega \cdot (n+1) + \sigma(n) + 3}$ carries a scale of length $\aleph_{\omega^2+2}$ such that, as in the previous item, all relevant points are very good.}
\item{If $\sigma \in {^\omega}\omega$ is such that, for all $n<\omega$, $\sigma(n)\geq n$, then $\prod_{n<\omega}\aleph_{\omega \cdot (n+1) + \sigma(n) + 3}$ carries a bad scale of length $\aleph_{\omega^2+1}$.}
\end{itemize}

Also, just as in Section 5, the bad scales in $\prod_{n<\omega}\aleph_{\omega \cdot (n+1) + \sigma(n) + 2}$ in $V[G*J]$ have stationarily many points of cofinality $\aleph_{\omega+1}$ where there are eubs $g$ such that, for all $n<\omega$, $\mathrm{cf}(g(n)) = \aleph_{\sigma(n)+1}$. Finally, as before, if we force over $V[G*J]$ with $\mathrm{Coll}(\aleph_{\omega^2+1}, \aleph_{\omega^2+2})$, then all scales in $V[G*J]$ remain scales (now all of length $\aleph_{\omega^2+1}$) in the further extension. The bad scales remain bad, the very good scales remain good, and all scales which were of length $\aleph_{\omega^2+2}$ in $V[G*J]$ are also very good in the further extension. 

\section{The Good Ideal}

We now turn to the third question raised by Cummings and Foreman in \cite{cummingsforeman}. We first give some background.

Let $\kappa$ be a singular cardinal, and let $A$ be a progressive set of regular cardinals cofinal in $\kappa$. Let $I$ be the collection of $B\subseteq A$ such that either $B$ is bounded or $\prod B$ carries a scale of length $\kappa^+$. $I$ is easily seen to be an ideal, and one of the seminal results of PCF Theory is the fact \cite{shelah} that $I$ is singly generated, i.e. there is $B^* \subseteq A$ such that, for all $B\in I$, $B\subseteq^* B^*$, where $\subseteq^*$ denotes inclusion modulo bounded sets. Thus, if $\kappa$ is a singular cardinal that is not a cardinal fixed point, then there is a largest subset of the regular cardinals below $\kappa$, modulo bounded sets, which carries a scale of length $\kappa^+$. Such a set is called the first PCF generator.

Cummings and Foreman asked whether, when the first PCF generator exists, there is also a maximal set, again modulo bounded sets, which carries a good scale. This is the case in the models obtained above by diagonal supercompact Prikry forcing, which were also the first known models in which this set is nontrivial and different from the first PCF generator itself, i.e. in which certain sets carry good scales and others carry bad scales. For example, in our final model in Section 6, after forcing with $\mathrm{Coll}(\aleph_{\omega^2+1}, \aleph_{\omega^2+2})$, the largest subset of regular cardinals below $\aleph_{\omega^2}$ that carries a good scale is, modulo bounded subsets, $\{\aleph_{\omega \cdot n + m} \mid n\leq \omega, m<n+2 \}$. Does such a set always exist? In this section, we will only be concerned with singular cardinals of countable cofinality and with scales in $\prod A$, where $\mathrm{otp}(A) = \omega$. With this in mind, we make the following definition.

\begin{definition}
Suppose $\kappa$ is a singular cardinal of countable cofinality. Then $I_{gd}[\kappa]$ is the collection of $A\subseteq \kappa$ such that $A$ is a set of regular cardinals and either $A$ is finite or $\mathrm{otp}(A) = \omega$ and $\prod A$ carries a good scale of length $\kappa^+$.
\end{definition}

It is easily seen that $I_{gd}[\kappa]$ is an ideal. A question related to that of Cummings and Foreman is whether or not $I_{gd}[\kappa]$ is a P-ideal, i.e. whether or not, given $\langle A_n \mid n<\omega \rangle$ such that, for all $n<\omega$, $A_n \in I_{gd}[\kappa]$, there exists $A\in I_{gd}[\kappa]$ such that, for all $n<\omega$, $A_n \subseteq^* A$. We do not answer this question, but we provide some partial results. In the first part of this section, we analyze individual good and bad points in scales, proving that there are consistently local obstacles to proving that $I_{gd}[\kappa]$ is a P-ideal (though we do not know whether this pathological behavior can consistently occur simultaneously at enough points to provide an actual counterexample to $I_{gd}[\kappa]$ being a P-ideal). In the second part of the section, we show that, after forcing with a finite-support iteration of Hechler forcing of length $\omega_1$, $I_{gd}[\kappa]$ is necessarily a P-ideal.

Let $\langle A_n \mid n<\omega \rangle$ be a sequence of elements of $I_{gd}[\kappa]$ and suppose, without loss of generality, that each $A_n$ is infinite. Then, by results of Shelah, there is $A\subseteq \kappa$ of order type $\omega$ such that $\prod A$ carries a scale of length $\kappa^+$ and, for all $n<\omega$, $A_n \subseteq^* A$. Again without loss of generality, by adjusting the $A_n$'s if necessary, we may assume that $A = \bigcup A_n$ and, for all $n<\omega$, $A_n \subseteq A_{n+1}$. 

Let $\overrightarrow{f} = \langle f_\alpha \mid \alpha < \kappa^+ \rangle$ be a scale in $\prod A$. For $B\subseteq A$, let $\overrightarrow{f}^B$ denote $\langle f_\alpha \restriction B \mid \alpha < \kappa^+ \rangle$. If $B$ is infinite, then $\overrightarrow{f}^B$ is a scale in $\prod B$. Also, since, for each $n<\omega$, $\prod A_n$ carries a good scale, we know that $\overrightarrow{f}^{A_n}$ is itself a good scale. Thus, for each $n<\omega$, there is a club $C_n \subseteq \kappa^+$ such that for every $\beta \in C_n$ of uncountable cofinality, $\beta$ is good for $\overrightarrow{f}^{A_n}$. Intersecting the clubs, there is a club $C\subseteq \kappa^+$ such that, for every $n<\omega$ and every $\beta \in C$ of uncountable cofinality, $\beta$ is good for $\overrightarrow{f}^{A_n}$.

We want to know if there is $B\subseteq A$ such that $\prod B$ carries a good scale and, for all $n<\omega$, $A_n \subseteq^* B$. We first take a local view and focus on individual good points. Suppose that $\beta \in C$ has uncountable cofinality. For $n<\omega$, let $g_n \in \prod A_n$ be an eub for $\overrightarrow{f}^{A_n} \restriction \beta$ such that, for all $i\in A_n$, $\mathrm{cf}(g_n(i)) = \mathrm{cf}(\beta)$. By uniqueness of eubs, if $m < n$, then $g_n \restriction A_m =^* g_m$. Define $g\in {^A\mathrm{On}}$ by letting $g\restriction A_0 = g_0$ and, for all $n<\omega$, $g\restriction (A_{n+1} \setminus A_n) = g_{n+1} \restriction (A_{n+1} \setminus A_n)$. Then we have that $g\restriction A_n =^* g_n$ for all $n<\omega$.

By uniqueness of eubs, if $B^* \subseteq A$, $A_n \subseteq^* B^*$ for all $n<\omega$, and $g^*$ is an eub for $\overrightarrow{f}^{B^*}\restriction \beta$, then, for all $n<\omega$, $g^* \restriction (A_n \cap B^*) =^* g_n \restriction (A_n \cap B^*)$. Thus, if there is such a $B^*$ and $g^*$, then there is a $B\subseteq A$ such that $A_n \subseteq^* B$ for all $n<\omega$ and $g\restriction B$ is an eub for $\overrightarrow{f}^B \restriction \beta$. We now turn to the question of when such a $B$ exists.

We note that a subset of $A$ that almost contains all of the $A_n$'s can be specified by an element of ${^\omega}\omega$. Namely, if each $A_n$ is enumerated in increasing order as $\langle i^n_k \mid k < \omega \rangle$ and $\sigma \in {^\omega}\omega$, let \[B_\sigma = \bigcup_{n<\omega} A_n \setminus i^n_{\sigma(n)}.\] Each $B_\sigma$ is a subset of $A$ that almost contains all of the $A_n$'s, and each subset of $A$ that almost contains all of the the $A_n$'s contains a set of the form $B_\sigma$.

With this in mind, it is not surprising that cardinal characteristics of the continuum have an impact on the situation. Recall that $\mathfrak{b}$, or the {\it bounding number}, is the size of the smallest family of functions that is unbounded in $({^\omega}\omega, <^*)$. $\mathfrak{d}$, or the {\it dominating number}, is the size of the smallest family of functions that is cofinal in $({^\omega}\omega, <^*)$.

For simplicity, when considering questions about the local behavior of scales, we will without loss of generality think of the $A_n$'s as being subsets of $\omega$ and consider increasing sequences of functions in ${^A}\mathrm{On}$.

\begin{lemma}
 Let $\lambda$ be a regular, uncountable cardinal. Let $\langle A_n \mid n<\omega \rangle$ be such that each $A_n$ is an infinite subset of $\omega$ and, for all $m < n$, $A_m \subseteq A_n$. Let $A=\bigcup_{n<\omega} A_n$. Let $\overrightarrow{f} = \langle f_\alpha \mid \alpha < \lambda \rangle$ be a $<^*$-increasing sequence of functions in ${^A}\mathrm{On}$ such that, for every $n<\omega$, $\overrightarrow{f}^{A_n} = \langle f_\alpha \restriction A_n \mid \alpha < \lambda \rangle$ has an eub of uniform cofinality $\lambda$. If either $\lambda < \mathfrak{b}$ or $\lambda > \mathfrak{d}$, then there is $B\subseteq A$ such that, for all $n$, $A_n \subseteq^* B$ and $\overrightarrow{f}^B$ has an eub of uniform cofinality $\lambda$.
\end{lemma}

\begin{proof}
For $n<\omega$, let $g_n$ be an eub of uniform cofinality $\lambda$ for $\overrightarrow{f}^{A_n}$. Define $g\in {^A}\mathrm{On}$ as before so that, for all $i\in A$, $\mathrm{cf}(g(i)) = \lambda$ and, for all $n<\omega$, $g\restriction A_n =^* g_n$.

For each $i\in A$, let $\langle \gamma^i_\xi \mid \xi < \lambda \rangle$ be increasing and cofinal in $g(i)$ and, for $\xi < \lambda$, define $h_\xi \in {^A\mathrm{On}}$ by $h_\xi(i) = \gamma^i_\xi$. 

We first consider the case $\lambda < \mathfrak{b}$. For each $\alpha < \lambda$ and each $n<\omega$, there is $\xi < \lambda$ such that $f_\alpha \restriction A_n <^* h_\xi \restriction A_n$. Thus, there is $\xi_\alpha < \lambda$ such that, for all $n<\omega$, $f_\alpha \restriction A_n <^* h_{\xi_\alpha} \restriction A_n$. Define $\sigma_\alpha \in {^\omega \omega}$ by letting $\sigma_\alpha(n)$ be the least $i$ such that $f_\alpha \restriction (A_n \setminus i) < h_{\xi_\alpha} \restriction (A_n \setminus i)$.

Similarly, for each $\xi < \lambda$ and each $n<\omega$, there is $\alpha < \lambda$ such that $h_\xi \restriction A_n <^* f_\alpha \restriction A_n$, so there is $\alpha_\xi < \lambda$ such that, for all $n<\omega$, $h_\xi \restriction A_n <^* f_{\alpha_\xi} \restriction A_n$. Define $\tau_\xi \in {^\omega \omega}$ by letting $\tau_\xi(n)$ be the least $i$ such that $h_\xi \restriction (A_n \setminus i) < f_{\alpha_\xi} \restriction (A_n \setminus i)$.

Since $\lambda < \mathfrak{b}$, we can find $\sigma \in {^\omega \omega}$ such that for all $\alpha, \xi < \lambda$, we have $\sigma_\alpha, \tau_\xi <^* \sigma$. Let $B = B_\sigma$. We claim that this $B$ is as desired. To see this, let $\alpha < \lambda$. $f_\alpha \restriction B_{\sigma_\alpha} < h_{\xi_\alpha} \restriction B_{\sigma_\alpha}$. But, since $\sigma_\alpha <^* \sigma$, we know that $B \subseteq^* B_{\sigma_\alpha}$, so $f_\alpha \restriction B <^* h_{\xi_\alpha} \restriction B$. By the same argument, for each $\xi < \lambda$, $h_\xi \restriction B <^* f_{\alpha_\xi} \restriction B$. Thus, $\langle h_\xi \restriction B \mid \xi < \lambda \rangle$ is cofinally interleaved with $\overrightarrow{f}^B$, so, by Proposition \ref{goodprop}, $g \restriction B$ is an eub for $\overrightarrow{f}^B$.

Now suppose that $\lambda > \mathfrak{d}$. Let $\mathcal{F}$ be a family of functions cofinal in ${^\omega}\omega$ such that $|\mathcal{F}| = \mathfrak{d}$. As above, for each $\alpha < \lambda$, we can find $\xi_\alpha < \lambda$ and $\sigma_\alpha \in {^\omega}\omega$ such that $f_\alpha \restriction B_{\sigma_\alpha} <^* h_{\xi_\alpha} \restriction B_{\sigma_\alpha}$, but we now also require that $\sigma_\alpha \in \mathcal{F}$. This is possible because $\mathcal{F}$ is cofinal in ${^\omega}\omega$. Since $\lambda > \mathfrak{d}$, there is $\sigma^* \in \mathcal{F}$ such that $\sigma_\alpha = \sigma^*$ for cofinally many $\alpha < \lambda$. But, since $\overrightarrow{f}$ is $<^*$-increasing, it is actually the case that, for every $\alpha < \lambda$, there is $\xi^*_\alpha < \lambda$ such that $f_\alpha \restriction B_{\sigma^*} <^* h_{\xi^*_\alpha} \restriction B_{\sigma^*}$.

Similarly, we can find $\tau^* \in \mathcal{F}$ such that, for every $\xi < \lambda$, there is $\alpha^*_\xi < \lambda$ such that $h_\xi \restriction B_{\tau^*} <^* f_{\alpha^*_\xi} \restriction B_{\tau^*}$. Let $\sigma > \sigma^*, \tau^*$, and let $B = B_\sigma$. It is easily seen as before that $\langle h_\xi \restriction B \mid \xi < \lambda \rangle$ is cofinally interleaved with $\overrightarrow{f}^B$, so $g \restriction B$ is an eub for $\overrightarrow{f}^B$.
\end{proof}

Now suppose that $\mathfrak{b} \leq \lambda \leq \mathfrak{d}$. Let $\langle A_n \mid n<\omega \rangle$ be a $\subseteq$-increasing sequence of infinite subsets of $\omega$, let $A = \bigcup A_n$, and let $\overrightarrow{f} = \langle f_\alpha \mid \alpha < \lambda \rangle$ be a $<^*$-increasing sequence of functions in ${^A}\mathrm{On}$ such that, for each $n<\omega$, $\overrightarrow{f}^{A_n}$ has an eub, $g_n$, of uniform cofinality $\lambda$. Let $g\in {^A}\mathrm{On}$ be of uniform cofinality $\lambda$ such that, for all $n<\omega$, $g\restriction A_n =^* g_n$. We would like to understand how $g$ can fail to be an eub. 

For each $i\in A$, let $\langle \gamma^i_\xi \mid \xi < \lambda \rangle$ be increasing and cofinal in $g(i)$ and, for $\xi < \lambda$, let $h_\xi \in {^A}\mathrm{On}$ be defined by $h_\xi(i) = \gamma^i_\xi$. For $\sigma \in {^\omega}\omega$, the statement that $g\restriction B_\sigma$ is an eub for $\overrightarrow{f}^{B_\sigma}$ is equivalent to the statement that $\langle h_\xi \restriction B_\sigma \mid \xi < \lambda \rangle$ is cofinally interleaved in $\overrightarrow{f}^{B_\sigma}$. There are two ways in which this could fail to happen:
\begin{enumerate}
\item{There is $\alpha < \lambda$ such that, for every $\xi < \lambda$, $f_\alpha \restriction B_\sigma \not<^* h_\xi \restriction B_\sigma$.}
\item{There is $\xi < \lambda$ such that, for every $\alpha < \lambda$, $h_\xi \restriction B_\sigma \not<^* f_\alpha \restriction B_\sigma$.}
\end{enumerate}
Thus, if $g\restriction B_\sigma$ fails to be an eub for $\overrightarrow{f}^{B_\sigma}$ for all $\sigma \in {^\omega}\omega$, one of the above two situations must occur for a $<^*$-cofinal set of ${^\omega}\omega$, so in fact one of them must occur for all $\sigma \in {^\omega}\omega$.

We first concentrate on case 1.

\begin{lemma}
Let $\sigma \in {^\omega}\omega$. Suppose $\alpha < \lambda$ is such that, for every $\xi < \lambda$, $f_\alpha \restriction B_\sigma \not<^* h_\xi \restriction B_\sigma$. Then $f_\alpha \restriction B_\sigma \not<^* g \restriction B_\sigma$.
\end{lemma}

\begin{proof}
Suppose for sake of contradiction that $f_\alpha \restriction B_\sigma <^* g\restriction B_\sigma$. Let $k<\omega$ be such that $f_\alpha(i) < g(i)$ for all $i\in B_\sigma \setminus k$. For $i \in B_\sigma \setminus k$, let $\xi_i < \lambda$ be such that $f_\alpha(i) < h_{\xi_i}(i)$. Let $\xi = \sup_{i\in B_\sigma \setminus k} \xi_i$. Then $f_\alpha(i) < h_\xi(i)$ for all $i\in B_\sigma \setminus k$, so $f_\alpha \restriction B_\sigma <^* h_\xi \restriction B_\sigma$.
\end{proof}

Thus, we are in case 1 if and only if $g\restriction B_\sigma$ is not in fact an upper bound for $\overrightarrow{f}^{B_\sigma}$. We now show that, if $\lambda = \mathfrak{b} = \omega_1$, this can hold simultaneously for all $\sigma \in {^\omega}\omega$. Without loss of generality, in what follows we assume that $A = \omega$.

\begin{lemma}
Suppose $\mathfrak{b} = \omega_1$. Let $\langle A_n \mid n<\omega \rangle$ be such that, for each $n$, $A_n \subseteq A_{n+1}$, $A_{n+1} \setminus A_n$ is infinite, and $\bigcup_{n<\omega} A_n = \omega$. There is a sequence of functions $\overrightarrow{f} = \langle f_\alpha \mid \alpha < \omega_1 \rangle$, $<^*$-increasing in ${^\omega}\mathrm{On}$, such that, for every $n<\omega$, $\overrightarrow{f}^{A_n}$ has an eub, $g_n$, of uniform cofinality $\omega_1$ but, letting $g$ be such that $g\restriction A_n =^* g_n$ for every $n<\omega$, for every $\sigma \in {^\omega}\omega$, $g \restriction B_\sigma$ is not an upper bound for $\overrightarrow{f}^{B_\sigma}$.
\end{lemma}

\begin{proof}
Fix a $<^*$-increasing, unbounded sequence $\langle \sigma_\alpha \mid \alpha < \omega_1 \rangle$ in ${^\omega}\omega$. We first construct a useful sequence of subsets of $\omega$.

\begin{claim}
There is a sequence $\langle X_\alpha \mid \alpha < \omega_1 \rangle$ such that, for every $\alpha < \beta < \omega_1$,
\begin{enumerate}
\item{$X_\alpha \subseteq \omega$ and $X_\alpha \subseteq^* X_\beta$.}
\item{For all $n<\omega$, $X_\alpha \cap A_n$ is finite.}
\item{For all $n<\omega$ $X_{\alpha+1} \cap (A_{n+1} \setminus (A_n \cup \sigma_{\alpha}(n+1)))$ is nonempty.}
\end{enumerate}
\end{claim}

\begin{proof}
We construct $\langle X_\alpha \mid \alpha < \omega_1 \rangle$ by recursion on $\alpha$. Let $X_0 = \emptyset$. Given $X_\alpha$, let $X_{\alpha+1} = X_\alpha \cup \{\min(A_{n+1} \setminus (A_n \cup \sigma_\alpha(n+1))) \mid n<\omega \}$. Requirement 1 is clearly satisfied. For each $n<\omega$, $|X_{\alpha+1} \cap A_n | \leq |X_\alpha \cap A_n| + n$, so requirement 2 is satisfied as well. Finally, since, for all $n$, $A_{n+1} \setminus A_n$ is infinite, $X_{\alpha+1}$ contains an element of $A_{n+1} \setminus (A_n \cup \sigma_{\alpha}(n+1))$, so requirement 3 is satisfied.

If $\beta < \omega_1$ is a limit ordinal, then we just need $X_\beta$ to satisfy requirements 1 and 2. Fix a bijection $\tau_\beta :\omega \rightarrow \beta$, and let \[X_\beta = \bigcup_{n<\omega}(X_{\tau_\beta(n)} \setminus \bigcup_{k<n}A_k).\]
For $\alpha < \beta$, if $\alpha = \tau_\beta(n)$, then, since $X_\alpha \cap A_k$ is finite for all $k$, $X_\alpha \subseteq^* X_\alpha \setminus (\bigcup_{k<n} A_k) \subseteq X_\beta$, so $X_\alpha \subseteq^* X_\beta$. Also, for all $n<\omega$, \[X_\beta \cap A_n \subseteq (\bigcup_{k\leq n} X_{\tau_\beta(k)}) \cap A_n,\] which is finite. Thus, $X_\beta$ satisfies 1 and 2 as desired, completing the construction.
\end{proof}

We now use the sequence $\langle X_\alpha \mid \alpha < \omega_1 \rangle$ to construct the desired sequence of functions $\overrightarrow{f}$. For $\alpha < \omega_1$ and $k<\omega$, let \[f_\alpha(k) = 
\begin{cases}
\alpha & \text{if } k\not\in X_\alpha \\ 
\omega_1 + \alpha & \text{if } k\in X_\alpha
\end{cases}\]

\begin{claim}
$\overrightarrow{f}$ is $<^*$-increasing.
\end{claim}

\begin{proof}
Let $\alpha < \beta < \omega_1$. Then $f_\beta(k) \leq f_\alpha(k)$ if and only if $k \in X_\alpha \setminus X_\beta$. But $X_\alpha \subseteq^* X_\beta$, so this only happens for finitely many values of $k$. Thus, $f_\alpha <^* f_\beta$.
\end{proof}

For an ordinal $\gamma$, let $c_\gamma$ denote the constant function with domain $\omega$ taking value $\gamma$ everywhere.

\begin{claim}
For every $n<\omega$, $c_{\omega_1} \restriction A_n$ is an eub for $\overrightarrow{f}^{A_n}$.
\end{claim}

\begin{proof}
For every $\alpha < \omega_1$, $X_\alpha \cap A_n$ is finite, so $f_\alpha \restriction A_n =^* c_\alpha \restriction A_n$. Since $c_{\omega_1} \restriction A_n$ is an eub for $\langle c_\alpha \restriction A_n \mid \alpha < \omega_1 \rangle$, it is also an eub for $\overrightarrow{f}^{A_n}$.
\end{proof}

\begin{claim}
For every $\sigma \in {^\omega}\omega$, $c_{\omega_1} \restriction B_\sigma$ is not an upper bound for $\overrightarrow{f}^{B_\sigma}$.
\end{claim}

\begin{proof}
Fix $\sigma \in {^\omega}\omega$, and find $\alpha < \omega_1$ such that $\sigma_\alpha \not<^* \sigma$. Then $Y = \{n \mid \sigma(n+1) \leq \sigma_\alpha(n+1) \}$ is infinite and, for every $n\in Y$, $X_{\alpha+1}$ contains an element of $A_{n+1} \setminus (A_n \cup \sigma_\alpha(n+1))$. Thus, $X_{\alpha+1} \cap B_\sigma$ is infinite. Since, for all $k\in X_{\alpha+1}$, $f_{\alpha+1}(k) > \omega_1$, we have $f_{\alpha+1} \restriction B_\sigma \not<^* c_{\omega_1} \restriction B_\sigma$, so $c_{\omega_1} \restriction B_\sigma$ is not an upper bound for $\overrightarrow{f}^{B_\sigma}$.
\end{proof}
Thus, for all $\sigma \in {^\omega}\omega$, $\overrightarrow{f}^{B_\sigma}$ has no eub.
\end{proof}

We now turn our attention to case 2 and show that it too is consistently possible.

\begin{lemma}
Suppose $\mathfrak{d} = \omega_1$. Let $\langle A_n \mid n<\omega \rangle$ be such that, for each $n$, $A_n \subseteq A_{n+1}$, $A_{n+1} \setminus A_n$ is infinite, and $\bigcup_{n<\omega} A_n = \omega$. There is a sequence of functions $\overrightarrow{f} = \langle f_\alpha \mid \alpha < \omega_1 \rangle$, $<^*$-increasing in ${^\omega}\mathrm{On}$ such that, for every $n<\omega$, $\overrightarrow{f}^{A_n}$ has an eub but, defining $g$ and $\langle h_\xi \mid \xi<\omega_1 \rangle$ as above, for every $\sigma \in {^\omega}\omega$, there is $\xi_\sigma < \omega_1$ such that, for every $\alpha < \omega_1$, $h_{\xi_\sigma} \restriction B_\sigma \not<^* f_\alpha \restriction B_\sigma$.
\end{lemma}

\begin{proof}
Fix a $<^*$-increasing sequence $\langle \sigma_\alpha \mid \alpha < \omega_1 \rangle$ cofinal in ${^\omega}\omega$. For $\alpha<\omega_1$, let $B_\alpha$ denote $B_{\sigma_\alpha}$. 

\begin{claim}
There are subsets of $\omega$, $\langle X^\beta_\alpha \mid \alpha < \beta < \omega_1 \rangle$, such that, letting $X^\beta = \bigcup_{\alpha<\beta}X^\beta_\alpha$,
\begin{enumerate}
\item{For all $\beta < \omega_1$ and all $n<\omega$, $X^\beta \cap A_n$ is finite.}
\item{For all $\beta < \omega_1$ and all $\alpha < \alpha' < \beta$, $X^\beta_\alpha \subseteq X^\beta_{\alpha'}$.}
\item{For all $\beta < \omega_1$ and all $\alpha < \beta$, $X^\beta_\alpha \cap B_\alpha$ is infinite.}
\item{For all $\beta < \beta' < \omega_1$, $\bigcup_{\alpha < \beta} (X^\beta_\alpha \setminus X^{\beta'}_\alpha)$ is finite.}
\end{enumerate}
\end{claim}

\begin{proof}
We construct $\langle X^\beta_\alpha \mid \alpha < \beta < \omega_1 \rangle$ by recursion on $\beta$. First, let $X^1_0 = \{\min(A_{n+1} \setminus (A_n \cup \sigma_0(n+1))) \mid n<\omega \}$.

Next, suppose $\beta' < \omega_1$ and we have constructed $\langle X^\beta_\alpha \mid \alpha < \beta \leq \beta' \rangle$ satisfying requirements 1-4 above. For $\alpha < \beta'$, let $X^{\beta'+1}_\alpha = X^{\beta'}_\alpha$, and let $X^{\beta'+1}_{\beta'} = X^{\beta'} \cup \{\min(A_{n+1} \setminus (A_n \cup \sigma_{\beta'}(n+1)))\mid n<\omega \}$. It is immediate by the inductive hypothesis that this still satisfies the requirements.

Finally, suppose $\beta' < \omega_1$ is a limit ordinal and that we have constructed $\langle X^\beta_\alpha \mid \alpha < \beta < \beta' \rangle$. Fix a bijection $\tau : \omega \rightarrow \beta'$. For $\alpha < \beta'$, denote $\tau^{-1}(\alpha)$ as $i_\alpha$. Note that, for a fixed $\alpha$, $\langle X^\beta_\alpha \mid \alpha < \beta < \beta' \rangle$ is $\subseteq^*$-decreasing, and each $X^\beta_\alpha$ has an infinite intersection with $B_{\alpha}$, so any finite subsequence has an infinite intersection contained in $B_{\alpha}$. For all $\alpha < \beta'$, we will first define $Y_\alpha = \{ y^\alpha_m \mid m< \omega \}$ by recursion on $m$. We will also define an increasing sequence of natural numbers, $\langle n^\alpha_m \mid m<\omega \rangle$.

Fix $\alpha < \beta'$. Let $n^\alpha_0$ be the least $n$ such that \[\bigcap_{\substack{j\leq i_\alpha \\ \tau(j)>\alpha}}X^{\tau(j)}_\alpha \cap (A_n \setminus \sigma_\alpha(n))\] is nonempty, and let $y^\alpha_0$ be an element of this intersection. Given $y^\alpha_m$ and $n^\alpha_m$, let $n^\alpha_{m+1}$ be the least $n>n^\alpha_m$ such that \[ \bigcap_{\substack{j\leq i_\alpha+m \\ \tau(j) > \alpha}}X^{\tau(j)}_\alpha \cap (A_n \setminus (\sigma_\alpha(n)\cup (y^\alpha_m+1))) \] is nonempty, and let $y^\alpha_{m+1}$ be an element of this intersection. Note that we can always find such an $n^\alpha_{m+1}$, since \[ \bigcap_{\substack{j\leq i_\alpha+m \\ \tau(j) > \alpha}}X^{\tau(j)}_\alpha \cap B_{\alpha}\] is infinite.

If $\alpha < \beta < \beta'$ and $i_\beta \leq i_\alpha+m$, $y^\alpha_k \in X^\beta_\alpha$ for all $k>m$, so $Y_\alpha \subseteq^* X^\beta_\alpha$ and, if $i_\beta < i_\alpha$, we actually have $Y_\alpha \subseteq X^\beta_\alpha$. From this, it follows that, for all $n<\omega$, $Y_\alpha \cap A_n$ is finite. Also, by construction, $Y_\alpha \cap B_{\alpha}$ is infinite. Now let \[Y^*_\alpha = Y_\alpha \setminus \bigcup_{n<i_\alpha}A_n.\] Note that $Y^*_\alpha$ has all of the properties of $Y_\alpha$ mentioned above. Finally, let \[X^{\beta'}_\alpha = \bigcup_{\gamma \leq \alpha} Y^*_\gamma.\]

We claim that this construction has succeeded. To see that requirement 1 is satisfied, note that for all $n<\omega$, $\{\alpha < \beta' \mid Y^*_\alpha \cap A_n \not= \emptyset \} \subseteq \tau``(n+1)$, which is finite, and, for each $\alpha < \beta'$, $Y^*_\alpha \cap A_n$ is finite, so $X^{\beta'} \cap A_n$ is the finite union of finite sets and hence finite.

Requirements 2 and 3 are obviously satisfied by the construction. To see 4, fix $\beta < \beta'$. If $\alpha < \beta$ is such that $i_\alpha > i_\beta$, then $Y^*_\alpha \subseteq X^\beta_\alpha$, so \[\bigcup_{\alpha < \beta} (X^\beta_\alpha \setminus X^{\beta'}_{\alpha}) \subseteq \bigcup_{\substack{\alpha < \beta \\ \i_\alpha < i_\beta}}(Y^*_\alpha \setminus X^\beta_\alpha),\] which is finite.
\end{proof}

We are now ready to define $\overrightarrow{f} = \langle f_\beta \mid \beta < \omega_1 \rangle$. First, if $\beta < \omega_1$ and $k\in X^\beta$, let $\alpha^\beta_k$ be the least $\alpha < \beta$ such that $k\in X^\beta_\alpha$. For $\beta < \omega_1$ and $k<\omega$, let \[f_\beta(k) = 
\begin{cases}
\omega_1 \cdot \beta & \text{if } k\not\in X^\beta \\ 
\omega_1 \cdot \alpha^\beta_k + \beta & \text{if } k\in X^\beta
\end{cases}\]

\begin{claim}
$\overrightarrow{f}$ is $<^*$-increasing.
\end{claim}

\begin{proof}
Let $\beta < \beta' < \omega_1$. If $k<\omega$, the only way we can have $f_{\beta'}(k) \leq f_\beta(k)$ is if there is $\alpha < \beta$ such that $k\in X^\beta_\alpha \setminus X^{\beta'}_\alpha$. By construction, there are only finitely many such values of $k$.
\end{proof}

\begin{claim}
For every $n<\omega$, $c_{\omega_1^2} \restriction A_n$ is an eub for $\overrightarrow{f}^{A_n}$.
\end{claim}

\begin{proof}
For each $\beta<\omega_1$ and $n<\omega$, $X^\beta \cap A_n$ is finite, so $f_\beta \restriction A_n =^* c_{\omega_1 \cdot \beta} \restriction A_n$. Since $c_{\omega_1^2}\restriction A_n$ is an eub for $\langle c_{\omega_1 \cdot \beta} \restriction A_n \mid \beta < \omega_1 \rangle$, it is also an eub for $\overrightarrow{f}^{A_n}$.
\end{proof}

\begin{claim}
For all $\alpha, \beta < \omega_1$, $c_{\omega_1 \cdot (\alpha+1)} \restriction B_{\alpha} \not<^* f_\beta \restriction B_{\alpha}$.
\end{claim}

\begin{proof}
If $\beta \leq \alpha$, then $f_\beta(k) < \omega_1 \cdot (\alpha+1)$ for all $k<\omega_1$, so $f_\beta < c_{\omega_1 \cdot (\alpha+1)}$. If $\beta > \alpha$, then $X^\beta_\alpha \cap B_{\sigma_\alpha}$ is infinite and, for all $k\in X^\beta_\alpha$, $f_\beta(k) < \omega_1 \cdot (\alpha+1)$. Thus, $c_{\omega_1 \cdot (\alpha+1)} \restriction B_{\sigma_\alpha} \not<^* f_\beta \restriction B_{\alpha}$.
\end{proof}
By the above claims, we can let $g = c_{\omega_1^2}$ and, for $\xi < \omega_1$, $h_\xi = c_{\omega_1 \cdot \xi}$. Let $\sigma \in {^\omega}\omega$. Since $\langle \sigma_\alpha \mid \alpha < \omega_1 \rangle$ is dominating, we can find $\alpha < \omega_1$ such that $\sigma <^* \sigma_\alpha$. Then $B_\sigma \subseteq^* B_{\alpha}$, so, by the last claim, for all $\beta < \omega_1$, $h_{\alpha+1} \restriction B_\sigma \not<^* f_\beta \restriction B_\sigma$.
\end{proof}

Thus, at individual points of scales we can have a situation in which $\beta$ is good for $\overrightarrow{f}^{A_n}$ for every $n<\omega$ but fails to be good for $\overrightarrow{f}^{B_\sigma}$ for every $\sigma \in {^\omega}\omega$. It remains open whether this can happen simultaneously at stationarily many points, thus providing a counterexample to $I_{gd}[\kappa]$ being a P-ideal. In fact, only something slightly weaker needs to happen to provide a counterexample, namely, for every $\sigma \in {^\omega}\omega$, there are stationarily many $\beta \in \kappa^+ \cap \mathrm{cof}(\geq \omega_1)$ such that $\beta$ is good for every $\overrightarrow{f}^{A_n}$ but fails to be good for $\overrightarrow{f}^{B_\sigma}$.

We now show that, starting in any ground model, a very mild forcing, namely a finite support iteration of Hechler forcing, forces $I_{gd}[\kappa]$ to be a P-ideal for every singular $\kappa$ of countable cofinality. This further suggests that the question under consideration perhaps has more to do with the structure of ${^\omega}\omega$ than with PCF-theoretic behavior at higher cardinals. 

We first recall the Hechler forcing notion. Conditions in the forcing poset are of the form $p = (s^p, f^p)$, where $s^p \in {^{n^p}\omega}$ for some $n^p < \omega$ and $f^p \in {^\omega}\omega$. $q\leq p$ if and only if 
\begin{enumerate}
\item{$n^q \geq n^p$} 
\item{$s^q \restriction n^p = s^p$}
\item{For every $k \in [n^p, n^q)$, $s^q(k)>f^p(k)$.}
\item{$f^q>f^p$.}
\end{enumerate}
Hechler forcing adds a dominating real, i.e. a $\sigma^* \in {^\omega}\omega$ such that, for every $\sigma \in ({^\omega}\omega)^V$, $\sigma <^* \sigma^*$.

We will need the following fact from \cite{abrahammagidor}.

\begin{fact} \label{goodstat}
Let $X$ be a set, let $I$ be an ideal on $X$, and let $\overrightarrow{f} = \langle f_\alpha \mid \alpha < \zeta \rangle$ be a $<_I$-increasing sequence of functions in ${^X}\mathrm{On}$. Suppose that $\lambda$ is such that $|X| < \lambda = \mathrm{cf}(\lambda) < \mathrm{cf}(\zeta)$. Then the following are equivalent.
\begin{enumerate}
\item{There are stationarily many $\beta \in \zeta \cap \mathrm{cof}(\lambda)$ such that $\overrightarrow{f} \restriction \beta$ has an eub of uniform cofinality $\lambda$.}
\item{$\overrightarrow{f}$ has an eub, $g$, such that, for all $x\in X$, $\mathrm{cf}(g(x))\geq \lambda$.}
\end{enumerate}
\end{fact}

\begin{lemma}
\label{Hech_1}
Let $\kappa$ be a singular cardinal of countable cofinality. For each $n<\omega$, let $A_n$ be an increasing $\omega$-sequence of regular cardinals, cofinal in $\kappa$, such that $\prod A_n$ carries a good scale. Let $\mathbb{P}$ be Hechler forcing. Then, in $V^\mathbb{P}$, there is an $\omega$-sequence $B\subset \kappa$ such that, for all $n<\omega$, $A_n \subseteq^* B$ and $\prod B$ carries a good scale.
\end{lemma}

\begin{proof}
In $V$, let $B_0$ be an $\omega$-sequence, cofinal in $\kappa$, such that $A_n \subseteq^* B_0$ for all $n<\omega$ and $\prod B_0$ carries a scale. Without loss of generality, we may assume that $B_0 = \bigcup_{n<\omega} A_n$ and that all elements of $B_0$ are uncountable. Also, for each $n<\omega$, enumerate $A_n$ in increasing order by $\langle i^n_k \mid k<\omega \rangle$. Let $\overrightarrow{f} = \langle f_\alpha \mid \alpha < \kappa^+ \rangle$ be a scale in $\prod B_0$. Then, for each $n<\omega$, $\overrightarrow{f}^{A_n}$ is a scale in $\prod A_n$ and, since each $\prod A_n$ carries a good scale, each $\overrightarrow{f}^{A_n}$ is good. Let $C_n$ be club in $\kappa^+$ such that, if $\beta \in C_n \cap \mathrm{cof}(\geq \omega_1)$, then $\beta$ is good for $\overrightarrow{f}^{A_n}$, and let $C = \bigcap_{n<\omega} C_n$.

Let $\beta \in C\cap \mathrm{cof}(\geq \omega_1)$ and, for each $n<\omega$, let $g^\beta_n \in \prod A_n$ be an eub for $\overrightarrow{f}^{A_n} \restriction \beta$ such that $\mathrm{cf}(g^\beta_n(i))=\mathrm{cf}(\beta)$ for all $i\in A_n$. Let $g_\beta \in \prod B_0$ be such that, for all $i\in B_0$, $\mathrm{cf}(g_\beta(i)) = \mathrm{cf}(\beta)$ and, for all $n<\omega$, $g_\beta \restriction A_n =^* g^\beta_n$.

For each $i\in B_0$, let $\langle \beta^i_\xi \mid \xi < \mathrm{cf}(\beta) \rangle$ be increasing and cofinal in $g_\beta (i)$. For $\xi < \mathrm{cf}(\beta)$, let $h^\beta_\xi \in \prod B_0$ be defined by $h^\beta_\xi(i) = \beta^i_\xi$. For $\xi < \xi' < \mathrm{cf}(\beta)$, we have $h^\beta_\xi < h^\beta_{\xi'} < g_\beta$.

For $\xi < \mathrm{cf}(\beta)$ and $n<\omega$, let $\alpha^n_\xi < \beta$ be such that $h^\beta_\xi \restriction A_n <^* f_{\alpha^n_\xi} \restriction A_n$. Such an ordinal exists, because $h_\xi < g_\beta$ and $g_\beta \restriction A_n$ is an eub for $\langle f_\alpha \restriction A_n \mid \alpha < \beta \rangle$. Let $a_\xi = \sup(\{a^n_\xi \mid n<\omega\})$. Define a function $\sigma^\beta_\xi \in {^\omega}\omega$ by letting $\sigma^\beta_\xi(n)$ be the least $j$ such that $h^\beta_\xi (i^n_k) < f_{\alpha_\xi}(i^n_k)$ for all $k\geq j$.

For $\alpha < \beta$ and $n<\omega$, $f_\alpha \restriction A_n <^* g_\beta \restriction A_n$, so, since $\mathrm{cf}(\beta) > \omega$, there is $\xi^n_\alpha < \mathrm{cf}(\beta)$ such that $f_\alpha \restriction A_n <^* h^\beta_{\xi^n_\alpha}\restriction A_n$. Let $\xi_\alpha = \sup(\{\xi^n_\alpha \mid n<\omega \})$. Define a function $\tau^\beta_\alpha \in {^\omega}\omega$ by letting $\tau^\beta_\alpha (n)$ be the least $j$ such that $f_\alpha (i^n_k) < h^\beta_{\xi_\alpha}(i^n_k)$ for all $k\geq j$.

Now let $G$ be $\mathbb{P}$-generic over $V$. Since $\mathbb{P}$ has the c.c.c., all cardinalities and cofinalities are preserved by $\mathbb{P}$.

\begin{claim}
In $V[G]$, $\overrightarrow{f}$ is still a scale in $\prod B_0$.
\end{claim}
\begin{proof}
$\overrightarrow{f}$ is clearly still $<^*$ increasing, so it remains to check that it is cofinal in $\prod B_0$. To this end, let $\dot{h}$ be a $\mathbb{P}$ name for a member of $\prod B_0$. For $i\in B_0$, let $X_i = \{\delta \mid \mbox{for some } p\in \mathbb{P}$, $p \Vdash ``\dot{h}(\check{i}) = \check{\delta}" \}$. By the c.c.c., each $X_i$ is countable, so we may define a function $h^* \in \prod B_0$ in $V$ by $h^*(i) = \sup(X_i)$. Then there is $\alpha < \kappa^+$ such that $h^* <^* f_\alpha$, so $\Vdash ``\dot{h} <^* \check{f}_\alpha "$.
\end{proof}

In $V[G]$, let $\sigma \in {^\omega}\omega$ be the real added by $G$. In particular, $\sigma$ dominates all reals in $V$. Let $B = \bigcup_{n<\omega} (A_n \setminus i^n_{\sigma(n)})$. Clearly, $B$ is an $\omega$-sequence cofinal in $\kappa$ such that, for all $n<\omega$, $A_n \subseteq^* B$.

\begin{claim}
For all $\beta \in C \cap \mathrm{cof}(\geq \omega_1)$, $\beta$ is good for $\overrightarrow{f}^B$.
\end{claim}

\begin{proof}
Fix $\beta \in C \cap \mathrm{cof}(\geq \omega_1)$. It suffices to show that $\langle h^\beta_\xi \restriction B \mid \xi < \mathrm{cf}(\beta) \rangle$ is cofinally interleaved with $\langle f_\alpha \restriction B \mid \alpha < \beta \rangle$. Fix $\alpha < \beta$, and let $\xi$ be the $\xi_\alpha$ used above in the definition of $\tau^\beta_\alpha$. $\tau^\beta_\alpha <^* \sigma$, so there is $m<\omega$ such that $\tau^\beta_\alpha(n) < \sigma(n)$ for all $n\geq m$. Thus, by the definition of $\tau^\beta_\alpha$, we know that for all $n\geq m$, $f_\alpha \restriction \bigcup_{n\geq m}(A_n \setminus i^n_{\sigma(n)}) < h^\beta_\xi \restriction \bigcup_{n\geq m}(A_n \setminus i^n_{\sigma(n)})$. Also, for each $n<m$, $f_\alpha \restriction A_n <^* h^\beta_\xi \restriction A_n$. Putting this together, we get $f_\alpha \restriction B <^* h^\beta_\xi \restriction B$. An identical argument shows that, for all $\xi < \mathrm{cf}(\beta)$, there is $\alpha < \beta$ such that  $h^\beta_\xi \restriction B <^* f_{\alpha} \restriction B$.
\end{proof}
Thus, $\overrightarrow{f}^B$ is a good scale in $\prod B$, and the proof is complete.
\end{proof}

\begin{theorem}
Let $\langle \mathbb{P}_\gamma \mid \gamma \leq \omega_1 \rangle$ be a finite-support iteration of Hechler forcing. Then, in $V^{\mathbb{P}_{\omega_1}}$, for every singular cardinal $\kappa$ of countable cofinality, $I_{gd}[\kappa]$ is a P-ideal.
\end{theorem}

\begin{proof}
Let $\mathbb{P} = \mathbb{P}_{\omega_1}$. Since every pair of conditions $p$ and $q$ in the Hechler poset such that $s^p = s^q$ are compatible, Hechler forcing is $\sigma$-centered. Thus, since $\mathbb{P}$ is a finite-support iteration of length $\omega_1$ of $\sigma$-centered forcings, $\mathbb{P}$ is itself $\sigma$-centered (and thus $\omega_1$-Knaster).

Let $G$ be $\mathbb{P}$-generic over $V$. For $\eta < \gamma \leq \omega_1$, let $\mathbb{P}_\gamma = \mathbb{P}_\eta * \mathbb{P}_{\eta \gamma}$, and let $G_\eta$ and $G_{\eta \gamma}$ be the generic filters induced by $G$ on $\mathbb{P}_\eta$ and $\mathbb{P}_{\eta \gamma}$, respectively.

Let $\kappa$ be a singular cardinal of countable cofinality, and, in $V[G]$, let $A$ be an $\omega$-sequence cofinal in $\kappa$ such that $\prod A$ carries a good scale. By chain condition, there is $\gamma < \omega_1$ such that $A\in V[G_\gamma]$.
\begin{claim}
$\prod A$ carries a good scale in $V[G_\gamma]$.
\end{claim}

\begin{proof}
Work in $V[G_\gamma]$. Let $\dot{\overrightarrow{f}} = \langle \dot{f}_\alpha \mid \alpha < \kappa^+ \rangle$ be a $\mathbb{P}_{\gamma \omega_1}$-name such that $\Vdash_{\mathbb{P}_{\gamma \omega_1}} ``\dot{\overrightarrow{f}}$ is a good scale in $\prod \check{A}"$. For $\alpha < \kappa^+$ and $i\in A$, let $Y_{\alpha, i} = \{\nu \mid \mbox{for some } p\in \mathbb{P}_{\gamma \omega_1}$, $p\Vdash ``\dot{f}_\alpha(i) = \check{\nu}" \}$. Since $\mathbb{P}_{\gamma \omega_1}$ has the c.c.c., each $Y_{\alpha, i}$ is countable.

Now define $\overrightarrow{g} = \langle g_\alpha \mid \alpha < \kappa^+ \rangle$ by recursion on $\alpha$ as follows. Let $g_0$ be an arbitrary function in $\prod A$. Given $g_\alpha$, let $g_{\alpha+1} \in \prod A$ be such that $g_{\alpha+1} > g_\alpha$ and, for all $i \in A$, $g_{\alpha+1}(i) > \sup(Y_{\alpha, i})$. If $\alpha < \kappa^+$ is a limit ordinal, let  $g_\alpha \in \prod A$ be such that $g_\beta <^* g_\alpha$ for all $\beta < \alpha$. This is easily accomplished by a standard diagonalization argument.

It is clear that $\overrightarrow{g}$ is $<^*$-increasing. To see that it is cofinal in $\prod A$, let $h\in \prod A$ be given. Since $\dot{\overrightarrow{f}}$ is forced to be a scale in $\prod A$ and by the c.c.c., there is $\alpha < \kappa^+$ such that $\Vdash_{\mathbb{P}_{\gamma \omega_1}} ``\check{h} <^* \dot{f}_\alpha"$. But $\Vdash_{\mathbb{P}_{\gamma \omega_1}} ``\dot{f}_\alpha < \check{g}_{\alpha+1}"$, so $h <^* g_{\alpha+1}$. Thus, $\overrightarrow{g}$ is a scale in $\prod A$ and, by previous arguments, it remains a scale in $V[G]$. Since $\prod A$ carries a good scale in $V[G]$, $\overrightarrow{g}$ must be good in $V[G]$.

\begin{subclaim}
If $\beta \in \kappa^+ \cap \mathrm{cof}(\geq \omega_1)$ is good for $\overrightarrow{g}$ in $V[G]$, then it is good in $V[G_\gamma]$.
\end{subclaim}

\begin{proof}
We proceed by induction on $\beta$ and split into two cases depending on the cofinality of $\beta$. First, suppose $\mathrm{cf}(\beta) = \omega_1$. In $V[G]$, there is $i \in A$ and $\langle \beta_\xi \mid \xi < \omega_1 \rangle$ witnessing that $\beta$ is good, i.e.
\begin{itemize}
\item{$\langle \beta_\xi \mid \xi < \omega_1 \rangle$ is increasing and cofinal in $\beta$.}
\item{For all $\xi < \xi' < \omega_1$ and all $j \in A \setminus i$, $g_\xi(j)<g_{\xi'}(j)$.}
\end{itemize}
 In $V[G_\gamma]$, find $p \in \mathbb{P}_{\gamma \omega_1}$ such that $p$ forces $\beta$ to be good for $\overrightarrow{g}$, and let $\dot{i}$ and $\langle \dot{\beta}_\xi \mid \xi < \omega_1 \rangle$ be names forced by $p$ to witness that $\beta$ is good. First, find $p^* \leq p$ and $i^* \in A$ such that $p^* \Vdash ``\dot{i} = i^*"$. For $\xi < \omega_1$, find $p_\xi \leq p^*$  and $\beta^*_\xi$ such that $p_\xi ``\Vdash \dot{\beta}_\xi = \beta^*_\xi"$. Since $\mathbb{P}_{\gamma \omega_1}$ is $\omega_1$-Knaster, there is an unbounded $X \subseteq \omega_1$ such that if $\xi, \xi' \in X$, then $p_\xi$ and $p_{\xi'}$ are compatible. Then $i^*$ and $\langle \beta^*_\xi \mid \xi \in X \rangle$ witness that $\beta$ is good for $\overrightarrow{g}$ in $V[G_\gamma]$.
 
 Now suppose $\mathrm{cf}(\beta) > \omega_1$. In $V[G]$, there is a club $C$ in $\beta$ of order type $\mathrm{cf}(\beta)$ such that if $\alpha \in C \cap \mathrm{cof}(\geq \omega_1)$, then $\alpha$ is good for $\overrightarrow{g}$. By chain condition, there is a club $D\subseteq C$ such that $D\in V[G_\gamma]$, and, by induction, every $\alpha \in D \cap \mathrm{cof}(\geq \omega_1)$ is good for $\overrightarrow{g}$ in $V[G_\gamma]$. Thus, by Fact \ref{goodstat}, in $V[G_\gamma]$, $\langle g_\alpha \mid \alpha < \beta \rangle$ has an eub, $h$, such that, for all $i \in A$, $\mathrm{cf}(h(i)) > \omega$. But then, by chain condition, $h$ remains an eub in $V[G]$, where $\beta$ is good for $\overrightarrow{g}$. Thus, $\mathrm{cf}(h(i)) = \mathrm{cf}(\beta)$ for all but finitely many $i \in A$, so $\beta$ is good for $\overrightarrow{g}$ in $V[G_\gamma]$.
\end{proof}
Since $\overrightarrow{g}$ is a good scale in $V[G]$, there is a club $E\subseteq \kappa^+$ such that for all $\beta \in E \cap \mathrm{cof}(\geq \omega_1)$, $\beta$ is good for $\overrightarrow{g}$. By chain condition, there is a club $E' \subseteq E$ such that $E' \in V[G_\gamma]$. The previous subclaim implies that for all $\beta \in E' \cap \mathrm{cof}(\geq \omega_1)$, $\beta$ is good for $\overrightarrow{g}$ in $V[G_\gamma]$. Thus, $\overrightarrow{g}$ is a good scale in $V[G_\gamma]$.
\end{proof}

We finally show that $I_{gd}[\kappa]$ is a P-ideal in $V[G]$. So, in $V[G]$, let $\langle A_n \mid n<\omega \rangle$ be such that, for all $n<\omega$, $A_n \in I_{gd}[\kappa]$. We can assume that none of the $A_n$'s is finite. By the previous claim, there is $\gamma < \omega_1$ such that, for all $n<\omega$, $A_n \in V[G_\gamma]$ and $\prod A_n$ carries a good scale in $V[G_\gamma]$. Then, by Lemma \ref{Hech_1}, in $V[G_{\gamma+1}]$ there is $B\in I_{gd}[\mu]$ such that, for all $n<\omega$, $A_n \subseteq^* B$. Arguments as before show the good scale in $\prod B$ in $V[G_\gamma]$ remains a good scale in $V[G]$, so $B\in I_{gd}[\kappa]$ in $V[G]$ as well. Thus, $I_{gd}[\kappa]$ is a P-ideal in $V[G]$.
\end{proof}

\bibliography{scaleref}

\begin{thebibliography}{10}

\bibitem{abrahammagidor}
U.~Abraham and M.~Magidor.
\newblock Cardinal arithmetic.
\newblock In {\em Handbook of set theory}, pages 1149--1227. Springer, 2010.

\bibitem{cummingsforeman}
J.~Cummings and M.~Foreman.
\newblock Diagonal prikry extensions.
\newblock {\em The Journal of Symbolic Logic}, 75(4):1383, 2010.

\bibitem{cfm}
J.~Cummings, M.~Foreman, and M.~Magidor.
\newblock Squares, scales and stationary reflection.
\newblock {\em Journal of Mathematical Logic}, 1(01):35--98, 2001.

\bibitem{canonicalstructure}
J.~Cummings, M.~Foreman, and M.~Magidor.
\newblock Canonical structure in the universe of set theory: Part two.
\newblock {\em Annals of Pure and Applied Logic}, 142(1):55--75, 2006.

\bibitem{cummingsmagidor}
J.~Cummings and M.~Magidor.
\newblock Martin’s maximum and weak square.
\newblock In {\em Proceedings of the American Mathematical Society}, volume
  139, pages 3339--3348, 2011.

\bibitem{foremanmagidor}
M.~Foreman and M.~Magidor.
\newblock A very weak square principle.
\newblock {\em The Journal of Symbolic Logic}, 62(1):175--196, 1997.

\bibitem{gitiksharon}
M.~Gitik and A.~Sharon.
\newblock On sch and the approachability property.
\newblock {\em Proceedings of the American Mathematical Society}, 136(1):311,
  2008.

\bibitem{jech}
T.~Jech.
\newblock Set theory, the third millennium, revised and expanded ed.
\newblock {\em Springer Monographs in Mathematics, Springer-Verlag, Berlin},
  2003.

\bibitem{kojman}
M.~Kojman and S.~Shelah.
\newblock The pcf trichotomy theorem does not hold for short sequences.
\newblock {\em Archive for Mathematical Logic}, 39(3):213--218, 2000.

\bibitem{shelah}
S.~Shelah.
\newblock {\em Cardinal arithmetic}.
\newblock Clarendon Press Oxford, 1994.

\bibitem{su}
D.~Sinapova and S.~Unger.
\newblock Scales at $\aleph_\omega$.
\newblock 2014.
\newblock Preprint.

\end{thebibliography}
\bibliographystyle{plain}

\end{document}